\documentclass[11pt,reqno]{amsart}
\usepackage[margin=0.7in]{geometry}
\usepackage{amsmath,amssymb,amsthm,graphicx,amsxtra, setspace}
\usepackage[utf8]{inputenc}
\usepackage{mathrsfs}
\usepackage{hyperref}
\usepackage{upgreek}
\usepackage{mathtools}
\usepackage[mathcal]{euscript}
\usepackage{xcolor}
\usepackage{upref}
\allowdisplaybreaks

\usepackage[pagewise]{lineno}

\newtheorem{theorem}{Theorem}[section]
\newtheorem{lemma}[theorem]{Lemma}
\newtheorem{proposition}[theorem]{Proposition}

\newtheorem{corollary}[theorem]{Corollary}
\newtheorem{definition}[theorem]{Definition}
\newtheorem{example}[theorem]{Example}
\newtheorem{remark}[theorem]{Remark}

\newtheorem{hypothesis}[theorem]{Hypothesis}

\let\originalleft\left
\let\originalright\right
\renewcommand{\left}{\mathopen{}\mathclose\bgroup\originalleft}
\renewcommand{\right}{\aftergroup\egroup\originalright}

\def\N{\mathbb{N}}
\def\V{V}

\def\u{\mathfrak{p}}
\def\v{\mathfrak{m}}
\def\w{\mathfrak{u}}
\def\f{\mathfrak{g}}
\def\y{\mathbf{Z}}
\def\x{\boldsymbol{\psi}}
\def\z{\mathfrak{q}}

\def\D{\mathrm{D}}
\def\V{\mathbb{V}}
\def\L{\mathbb{L}}
\def\H{\mathbb{H}}
\def\P{\mathrm{P}}
\def\A{\mathcal{A}}
\def\Arm{\mathrm{A}}
\def\I{\mathrm{I}}
\def\B{\mathcal{B}}
\def\J{\mathcal{J}}
\def\K{\mathcal{K}}
\def\d{\mathrm{d}}

\def\wi{\widetilde}
\def\Tr{\mathrm{Tr}}
\def\X{\mathbb{X}}
\def\S{\mathcal{S}}
\def\W{\mathrm{W}}

\def\diver{\mathrm{div}}

\numberwithin{equation}{section} \allowdisplaybreaks

\newcommand{\R}{\mathbb{R}}

\renewcommand{\d}{\/\mathrm{d}\/}

\newcommand{\Addresses}{{
		\footnote{
			\noindent \textsuperscript{1}Center for Mathematics and Applications (NOVA Math), NOVA School of Science and Technology (NOVA FCT),	Portugal.\par\nopagebreak
			\noindent 
			
			\textit{e-mail:} \texttt{Kush Kinra: kushkinra@gmail.com, k.kinra@fct.unl.pt.}
			
			\noindent \textsuperscript{*}Corresponding author.
			
			\textit{Key words:} Deterministic and stochastic third-grade fluids, small forcing intensity, singleton attractor, upper semicontinuity, lower semicontinuity.
			
			Mathematics Subject Classification (2020): Primary 35B41, 35Q35; Secondary 37L55, 37N10, 35R60.

}}}

\begin{document}
	
	\title[Random attractors converging towards singleton attractor]{Rate of convergence of random attractors towards deterministic singleton attractor for a class of non-Newtonian fluids of differential type
		\Addresses}
	
	\author[Kush Kinra]
	{Kush Kinra\textsuperscript{1*}}

	\maketitle

	\begin{abstract}

In this article, we investigate the long-term dynamics of a class of two- and three-dimensional non-Newtonian fluids of differential type, known as third-grade fluids. We first show that when the external forcing is sufficiently small, the global attractor of the underlying system (which characterizes its asymptotic behavior) reduces to a single point. We then consider the system under stochastic perturbations, specifically infinite-dimensional additive white noise. In this random setting, we do not find conclusive evidence that the corresponding random attractor remains a single point, as in the deterministic case. However, we are able to estimate the rate at which the random attractor approaches the deterministic singleton attractor as the intensity of the stochastic noise tends to zero.

  	\end{abstract}


\section{Introduction}

From a physical standpoint, the concept of an attractor provides a fundamental explanation of how complex fluid flows evolve and stabilize over time. Although the equations governing fluid motion, such as the Navier-Stokes or non-Newtonian flow equations, are highly nonlinear and sensitive to initial conditions, real fluid systems often display a tendency to approach a statistically steady or recurrent state after transient effects decay. This long-term behavior reflects the balance between energy input from external forcing and energy dissipation due to viscosity. The attractor represents the mathematical manifestation of this physical stabilization process, describing the set of all possible states toward which the flow eventually evolves. Studying attractors therefore helps to identify and quantify the persistent structures, coherent patterns, or steady regimes that govern fluid motion at large times. Moreover, the finite-dimensional nature of attractors in many dissipative systems implies that, despite the infinite-dimensional character of the governing equations, the essential dynamics of the flow can be captured by a relatively small number of dominant modes, an observation that aligns closely with experimental and computational findings in fluid mechanics and turbulence theory (see \cite{Foias+Manley+Rosa+Temam_2001,Robinson2,R.Temam} etc.). Understanding attractors thus provides deep physical insight into how complex flows self-organize, lose memory of initial conditions, and exhibit predictable long-term behavior despite underlying nonlinearities.

Extensive research on stochastic perturbations of evolution equations has led to the development of the theory of random dynamical systems and random attractors (see \cite{Arnold}). The concept of a pathwise pullback random attractor, which can also be forward attracting in probability, was first introduced in \cite{CF,FS} and further studied in \cite{HCJA}. Since then, the existence and properties of random attractors for various stochastic partial differential equations (SPDEs) have been established in numerous works, see \cite{Han+Zhou_2025,Kinra+Cipriano_STGF,Kinra+Mohan_2025_JDDE,Wu+Nguyen+Bai} etc. for some recent works. Moreover, \cite{CCLR} highlighted the distinct impacts that different types of stochastic perturbations can have on the long-term behavior of deterministic systems.

An important characteristic of attractors is their stability under stochastic perturbations. This property refers to the behavior of the random attractor $\mathscr{A}_{\varsigma}(\omega)$ as the intensity of the stochastic noise $\varsigma$ tends to zero, specifically, whether $\mathscr{A}_{\varsigma}(\omega)$ converges to the corresponding deterministic global attractor $\mathscr{A}$.  In a metric space $(\mathbb{X}, d)$, an attractor is typically a compact, invariant set that attracts a certain family of subsets of the state space, most often all bounded sets. Such an attractor represents the collection of all possible asymptotic states of the system, and it is not necessarily a single point (singleton). For the analysis of stability, two main notions of convergence between attractors are commonly considered, and they are defined as follows:
\begin{itemize}
	\item [(i)] Upper semicontinuity, that is,
	\begin{align*}
		\lim_{\varsigma\to0^{+}} \text{dist}\left(\mathscr{A}_{\varsigma}(\omega),\mathscr{A}\right)=0.
	\end{align*}
	\item [(ii)] Lower semicontinuity, that is,
	\begin{align*}
		\lim_{\varsigma\to0^{+}} \text{dist}\left(\mathscr{A},\mathscr{A}_{\varsigma}(\omega)\right)=0,
	\end{align*}
\end{itemize}
where $\text{dist}(\cdot,\cdot)$ denotes the Hausdorff semi-distance between two non-empty subsets of $\mathbb{X}$, that is, for non-empty sets $E,F\subset \mathbb{X}$ $$\text{dist}(E,F)=\sup_{x\in E}\inf_{y\in F} d(x,y).$$ 

The upper semicontinuity of random attractors is generally easier to establish than lower semicontinuity. It relies on general properties of dynamical systems and attractors, with relatively mild assumptions (see \cite{Cao+Gao_2025,Chen+Tian+Zhang_2025,Kinra+Mohan_2023_DIE,Kinra+Mohan_2023_JDDE,Wang+Li+Yang+Jia_2021} etc.). The notion of upper semicontinuity for random attractors was first introduced in \cite{CLR} and later extended to non-compact random dynamical systems in \cite{Wang_2009}. In contrast, proving lower semicontinuity typically requires deeper analysis of the deterministic attractor’s structure or uniform attraction properties of the family of perturbed random attractors (see \cite{BV,CLR1,HR1,LK}). In this work, we investigate how fast random attractors, under additive white noise perturbation, converge to the deterministic (unperturbed) attractor, in the specific case where the latter consists of a single element.

\medskip

It is worth emphasizing that, although much of the literature focuses on Newtonian fluids governed by the Navier-Stokes equations, many real-world industrial and biological flows (see \cite{DR95,FR80,yas-fer_JNS} and references therein) do not follow Newton’s law of viscosity and therefore cannot be accurately described by these models. Such fluids often display complex rheological features, such as shear-thinning, shear-thickening, or viscoelastic behavior, that classical Newtonian's law fail to capture. As a result, more sophisticated mathematical models are required to describe their dynamics and predict their behavior in realistic settings. In recent years, non-Newtonian viscoelastic fluids of differential type have attracted significant attention (see, for example, \cite{Cioran2016}). In particular, third-grade fluid models have been employed in a number of simulation studies to better understand the behavior of nanofluids (see for instance \cite{PP19,RHK18}). Nanofluids are engineered suspensions of nanoparticles in a base fluid, such as water, oil, or ethylene glycol, and are known to exhibit enhanced thermal conductivity compared to the base fluid alone, making them highly relevant in technological and microelectronic applications. Therefore, a rigorous mathematical analysis of third-grade fluid equations is essential for understanding the behavior of these fluids.

We now briefly outline the derivation of the governing equations for non-Newtonian fluids of differential type. For a detailed discussion of the kinematics of such fluids, we refer the reader to \cite{Cioran2016}. Let $\v$ denote the velocity field of the fluid, and introduce the Rivlin-Ericksen kinematic tensors $\Arm_n$ for $n \geq 1$ (see \cite{RE55}), defined by
\begin{align*}
	\Arm_1(\v)&=\nabla \v+(\nabla \v)^T;	\,\Arm_n(\v)=\dfrac{\D}{\D t} \Arm_{n-1}(\v)+\Arm_{n-1}(\v)(\nabla \v)+(\nabla \v)^T\Arm_{n-1}(\v), \ n=2,3,\cdots,
\end{align*}
where $
\dfrac{\mathrm{D}}{\mathrm{D}t}\equiv
\dfrac{\partial}{\partial t}+(\v\cdot \nabla)$ is material derivative.

The constitutive law of fluids of grade $n$ reads $\mathbb{T}=-p\mathrm{I} + F(\Arm_1,\cdots,\Arm_n),$ where $\mathbb{T}$ is the Cauchy stress tensor, $p$ is the pressure and $F$ is an isotropic polynomial function of degree $n$, subject to the usual	requirement of material frame indifference,	see	\cite{Cioran2016}. The constitutive law of   third-grade fluids	$(n=3)$ is given by the following equation	 

For fluids of grade $n$, the constitutive relation is $\mathbb{T}=-p\mathrm{I} + F(\Arm_1,\cdots,\Arm_n),$ where $\mathbb{T}$ is the Cauchy stress tensor, $p$ is the pressure and $F$ is an isotropic polynomial function of degree $n$, satisfying the principle of material frame indifference (see \cite{Cioran2016}).
For third-grade fluids $(n=3)$, this constitutive law can be written as follows:
\begin{align*}
	\mathbb{T}=-p\mathrm{I}+\nu \Arm_1+\alpha_1\Arm_2+\alpha_2\Arm_1^2+\beta_1 \Arm_3+\beta_2(\Arm_1\Arm_2+\Arm_2\Arm_1)+\beta_3\Tr(\Arm_1^2)\Arm_1,
\end{align*}
where $\nu$ is the viscosity and $\alpha_1, \alpha_2, \beta_1, \beta_2, \beta_3$ are material moduli. Recall that, by Newton’s second law, the momentum equations are given by
$$\dfrac{\mathrm{D}\v}{\mathrm{D}t}=
\dfrac{\partial\v}{\partial t}+(\v\cdot \nabla) \v=\text{div}(\mathbb{T}).$$ 
If $\beta_1=\beta_2=\beta_3=0$, the constitutive equations reduce to those of second-grade fluids. It has been demonstrated that the Clausius–Duhem inequality, together with the requirement that the Helmholtz free energy attain a minimum in equilibrium, imposes the following conditions on the viscosity and material moduli:
\begin{align}\label{secondlaw}
	\nu \geq 0,\quad \alpha_1+\alpha_2=0, \quad \alpha_1\geq 0. 
\end{align}
Although second-grade fluids are mathematically easier to handle, rheologists working with various non-Newtonian fluids have not verified the restrictions in \eqref{secondlaw}. They have therefore concluded that the fluids tested are not truly second-grade fluids but instead possess a different constitutive structure. For further details, see \cite{FR80} and the references therein. Following \cite{FR80}, to ensure that the motion of the fluid is consistent with thermodynamics, one must impose that
\begin{equation}\label{third-grade-paremeters}
	\nu \geq 0, \quad \alpha_1\geq 0, \quad |\alpha_1+\alpha_2 |\leq \sqrt{24\nu\beta}, \quad \beta_1=\beta_2=0, \beta_3=\beta \geq 0.
\end{equation}	
Hence, the incompressible third-grade fluid equations read as
\begin{equation}
	\label{third-grade-fluids-equations}
	\left\{\begin{aligned}
		\partial_t(z(\v))-\nu \Delta \v+(\v\cdot \nabla)z(\v) +\displaystyle\sum_{j=1}^d[z(\v)]_j\nabla \v_j & -(\alpha_1+\alpha_2)\text{div}((\Arm(\v))^2)\\ -\beta \text{div}[\Tr(\Arm(\v)\Arm(\v)^T)\Arm(\v)] + \nabla \mathbf{P}
		& =  \f,  \\
		\diver (\v) & =  0, \\
		z(\v)&:=\v-\alpha_1\Delta \v, \\[0.2cm]
		\Arm(\v) & := \nabla \v+(\nabla \v)^T,
	\end{aligned}\right.
\end{equation}
where the viscosity $\nu$ and the material moduli	 $\alpha_1,\alpha_2$, $\beta$ 	verify	\eqref{third-grade-paremeters}.  Note that setting $\alpha_1=\alpha_2=0$	and $\beta=0$ recovers the classical Navier-Stokes equations (NSEs). Mathematically, $n$-grade fluids form a hierarchy of increasing complexity; compared to Newtonian (grade 1) and second-grade fluids, third-grade fluids involve more nonlinear terms and require a more intricate analysis.

Let $\mathfrak{D}$ be an open and connected subset (may be bounded or unbounded) of $\R^d$, $d\in\{2,3\}$, the boundary of which is sufficiently smooth. In this article,  we consider the system \eqref{third-grade-fluids-equations} under a subset of the physical conditions \eqref{third-grade-paremeters}, specifically
\begin{equation}\label{third-grade-paremeters-res}
	\nu > 0, \quad \alpha_1= 0, \ \alpha_2=\alpha, \quad |\alpha | < \sqrt{2\nu\beta}, \quad \beta_1=\beta_2=0, \beta_3=\beta > 0,
\end{equation}	
in the presence of an autonomous forcing term $\f$ and subject to Dirichlet boundary conditions, the equations take the form
\begin{equation}\label{equationV1}
	\left\{		\begin{aligned}
		\dfrac{\partial	\v}{\partial t} - \nu \Delta \v+(\v\cdot \nabla)\v-\alpha\text{div}((\Arm(\v))^2) - \beta \text{div}(|\Arm(\v)|^2\Arm(\v))+ \nabla \textbf{P}&=\f, &&   \text{in }  \mathfrak{D} \times (0,\infty),\\
		\text{div}\; \v&=0 \quad &&   \text{in } \mathfrak{D} \times [0,\infty),\\
		\v &= \boldsymbol{0} &&   \text{on } \partial\mathfrak{D}\times [0,\infty),\\
		\v(x,0)&=\x  \quad &&  \text{in } \mathfrak{D},
	\end{aligned}\right.
\end{equation}
where $\v:\mathfrak{D}\times [0,\infty) \to \mathbb{R}^d$,  $\mathbf{P}:\mathfrak{D}\times   [0,\infty) \to \mathbb{R}$ and	 $\f:\mathfrak{D}\times[0,\infty)\to \R^d $ represent the velocity field, pressure	an	external force, respectively.

We now turn to the literature concerning system \eqref{equationV1}, i.e., system \eqref{third-grade-fluids-equations} with $\alpha_1=0$. A natural starting point is the seminal work of Ladyzhenskaya \cite{Ladyzhenskaya67}, who introduced a new model for incompressible viscous fluids in which the viscosity depends on the velocity gradient. This system features nonlinear terms similar to those in \eqref{equationV1}. In \cite{Hamza+Paicu_2007}, the authors established global well-posedness for system \eqref{equationV1} in $\R^3$, assuming divergence-free initial data in $\L^2(\R^3)$. Their approach relied on a monotonicity method along with condition \ref{third-grade-paremeters-res}. The stochastic counterpart of system \eqref{equationV1} on bounded domains was treated in \cite{yas-fer_JNS}, where the authors proved well-posedness and also constructed an ergodic invariant measure. More recently, in \cite{Kinra+Cipriano_TGF}, the author together with Cipriano proved the existence of a unique pullback attractor for system \eqref{equationV1} on Poincar\'e domains (both bounded and unbounded). They subsequently extended this work in \cite{Kinra+Cipriano_STGF} by proving the existence of random attractors for the stochastic version of system \eqref{equationV1} with infinite-dimensional additive white noise, again on Poincaré domains. It is worth noting that \cite{Kinra+Cipriano_TGF, Kinra+Cipriano_STGF} impose the following assumption on the domains, and we adopt the same assumption in the present work.
\begin{hypothesis}\label{assumpO}
	There exists a positive constant $\lambda $ such that the following Poincar\'e inequality  is satisfied:
	\begin{align}\label{Poin}
		\lambda\int_{\mathfrak{D}} |\psi(x)|^2 \d x \leq \int_{\mathfrak{D}} |\nabla \psi(x)|^2 \d x,  \ \text{ for all } \  \psi \in \H^{1}_0 (\mathfrak{D}).
	\end{align}
\end{hypothesis}

\begin{example}
	A typical example of a unbounded Poincar\'e domain in $\R^2$ and $\mathbb{R}^3$ are $\mathfrak{D}=\R\times(-L,L)$ and $\mathfrak{D}=\R^2\times(-L,L)$ with $L>0$, respectively, see \cite[p.306]{R.Temam} and \cite[p.117]{Robinson2}. 
\end{example}
\begin{example}
	For any bounded domain $\mathfrak{D}$, Poincar\'e inequality \eqref{Poin} is satisfied with $\lambda=\lambda_1$, where $\lambda_1$ is the first eigenvalue of the Stokes operator defined on bounded domains. 
\end{example}

The goals of this work are twofold. First, we establish the existence of a deterministic singleton attractor for a class of two- and three-dimensional third-grade fluid equations on the domains satisfying Hypothesis \ref{assumpO}. Second, we show that random attractors for stochastic version of system \eqref{equationV1} under infinite-dimensional additive noise perturbation on the domains satisfying Hypothesis \ref{assumpO} converge to this deterministic singleton attractor as the noise intensity tends to zero. For sufficiently small forcing (see \eqref{C_2} below), the deterministic global attractor is shown to be a singleton. However, even when the deterministic attractor is a singleton, the corresponding random attractors are typically not. The reason is that the randomness acts as an additional forcing that prevents effective control of the overall forcing intensity (see \cite{HCPEK} for details). In this sense, the noise disrupts the singleton structure of the unperturbed attractor. Moreover, we prove that, under infinite-dimensional additive white noise and condition \eqref{C_2}, the random attractors $\mathscr{A}_{\varsigma}(\omega)$ converge to the deterministic singleton attractor $\mathscr{A}$ both upper and lower semi continuously, i.e.,
$$\text{dist}_{\H}\left(\mathscr{A}_{\varsigma}(\omega),\mathscr{A}\right)\leq\updelta_{\varsigma}(\omega)$$
where $$\text{dist}_{\H}(A,B)=\max\{\text{dist}(A,B),\text{dist}(B,A)\},$$ and each $\updelta_{\varsigma}$ is a random variable satisfying $\updelta_{\varsigma}(\omega)\sim \varsigma^{\delta}$ as $\varsigma\to 0^{+}$ for some $\delta>0$.  For the Newtonian case, the existence of singleton attractors for the 2D deterministic Navier-Stokes equations and the 2D and 3D convective Brinkman-Forchheimer equations under small forcing, along with the convergence of random attractors to the deterministic singleton attractor, has been discussed in \cite{HCPEK} and \cite{Kinra+Mohan_2022_EECT}, respectively.

\begin{remark}
	Note that both articles \cite{HCPEK} and \cite{Kinra+Mohan_2022_EECT}, which address Newtonian fluids, are set on the torus. Since any vector field $\psi\in\H^1(\mathbb{T}^d)$ with zero-mean satisfies the Poincar\'e inequality \eqref{Poin}, our results remain valid on the torus as well. Only minor adjustments to the functional spaces are needed to account for the zero-mean condition.
\end{remark}

The remainder of the paper is organized as follows. In the next section, we introduce the functional spaces required for our analysis. We then define the operators used to rewrite \eqref{equationV1} in an abstract form and establish several estimates that will be used later. The abstract formulation of \eqref{equationV1} and the corresponding solvability results are also presented in this section.
In Section \ref{sec3}, we impose a smallness condition on the external forcing $\f$ (see \eqref{C_2} below). Under this assumption, Theorem \ref{D-SA} establishes the existence of a singleton attractor for system \eqref{equationV1}.
In the final section, we study system \eqref{equationV1} perturbed by infinite-dimensional additive white noise. In Theorem \ref{Conver-add} under the smallness assumption \ref{C_2}, we prove that the associated random attractors converge to the deterministic singleton attractor with convergence rate $\varsigma^{\frac23}$. We also note that this rate is slower than the one known for the 2D Navier-Stokes equations (see Remark \ref{rem-ROC}) below.

\section{Mathematical formulation and preliminary results}\label{Sec2} \setcounter{equation}{0}

In this section, we discuss necessary function spaces and operators to obtain the abstract formulation for the system \eqref{equationV1}. In addition, we also provide the known results for the well-posedness of system \eqref{equationV1}.

\subsection{Notations and the functional setting }
Let	$m\in	\mathbb{N}^*:=\N\cup\{\infty\}$	and	$1\leq	p<	\infty$,	we	denote	by	$\mathbb{W}^{m,p}(\mathfrak{D})$ (resp. $\mathrm{W}^{m,p}(\mathfrak{D})$)	the	standard Sobolev space of matrix/vector-valued (resp. scalar-valued) functions	whose weak derivative up to order $m$	belong	to	the	Lebesgue	space	$\L^p(\mathfrak{D})$ (resp. $\mathrm{L}^p(\mathfrak{D})$) and set	$\H^m(\mathfrak{D})=\mathbb{W}^{m,2}(\mathfrak{D})$	and	$\H^0(\mathfrak{D})=\L^2(\mathfrak{D})$.

Let us first define the space $\mathscr{V}:=\{\v \in (\mathrm{C}^\infty_c(\mathfrak{D}))^d \,: \text{ div}(\v)=0\}$. Then, we denote the spaces $\H$, $\V$, ${\L}^{p}_{\sigma}$ ($1<p<\infty$) and $\V_s$ ($s>1$) as the closure of $\mathscr{V}$ in $\L^2(\mathfrak{D})$, $\H^1(\mathfrak{D})$, $\L^p(\mathfrak{D})$ ($1<p<\infty$) and $\H^s(\mathfrak{D})$ ($s>1$), respectively. For the characterization of spaces, we refer readers to \cite[Chapter 1]{Temam_1984}.

Next, let us introduce  the scalar product between two matrices $A:B=\Tr(AB^T)$
and denote $\vert A\vert^2:=A:A.$
The divergence of a  matrix $A\in \mathcal{M}_{d\times d}(E)$ is given by 
$\left\{\text{div}(A)_i\right\}_{i=1}^{i=d}=\left\{\displaystyle\sum_{j=1}^d\partial_ja_{ij}\right\}_{i=1}^{i=d}. $
The space $\H$ is endowed with the  $\L^2$-inner product $(\cdot,\cdot)$ and the associated norm $\Vert \cdot\Vert_{2}$.  In view of the Assumption \ref{assumpO}, on the functional space $\V$, we will consider  the following inner product $(\u,\v)_{\V}:= (\nabla	\u,\nabla	\v),$ and denote by $\Vert \cdot\Vert_{\V}$ the corresponding norm. The usual norms on the classical Lebesgue and Sobolev spaces $\mathrm{L}^p(\mathfrak{D})$ (resp. $\mathbb{L}^p(\mathfrak{D})$) and $\mathrm{W}^{m,p}(\mathfrak{D})$ (resp. $\mathbb{W}^{m,p}(\mathfrak{D})$) will be denoted by $\|\cdot \|_p$ and 
$\|\cdot\|_{\mathrm{W}^{m,p}}$ (resp. $\|\cdot\|_{\mathbb{W}^{m,p}}$), respectively.
In addition, given a Banach space $E$, we will denote by $E^\prime$ its dual.

	Let	us	introduce	the	following	Banach	space	$(\X,\Vert	\cdot\Vert_{\X})$
$$ \X=\mathbb{W}_0^{1,4}(\mathfrak{D})\cap \V,	\quad	\text{	with	\;\;	} \Vert	\cdot\Vert_\X:=\Vert	\cdot\Vert_{\mathbb{W}^{1,4}}+\|\cdot\|_{\V}.$$
Indeed,	we	recall	that	$\mathbb{W}_0^{1,4}(\mathfrak{D})$	endowed	with $\Vert	\cdot\Vert_{\mathbb{W}^{1,4}}$-norm	is a Banach	space,	where		$$\Vert	\w \Vert_{\mathbb{W}^{1,4}}^4=\int_{\mathfrak{D}}	\vert	\w(x)\vert^4 \d x+\int_{\mathfrak{D}} \vert	\nabla	\w(x)\vert^4 \d x.$$
We denote	by	$\langle	\cdot,\cdot\rangle$ the duality pairing between $\X^{\prime}$ and $\X$.

\begin{lemma}
	In this lemma, we recall a couple to inequalities which will be used in the sequel:
	\begin{itemize}
		\item  There exists a constant $C_{S,d}>0$ (\cite[Subsection 2.4]{Kesavan_1989}) such that 
		\begin{align}\label{Sobolev-embedding3}
			\|\w\|_{\infty} \leq C_{S,d} \|\nabla\w\|_{4}, \quad \text{	for all } \;	\w\in	\mathbb{W}_0^{1,4}(\mathfrak{D}).
		\end{align}
	
	\item  There	exists a constant $C_{K,d}>0$	such	that (a Korn-type inequality, see	\cite[Theorem 2 (ii)]{DiFratta+Solombrino_Arxiv} and \cite[Theorem 4, p.90]{Kondratev+Oleinik_1988})
	\begin{align}\label{Korn-ineq}
		\Vert	\nabla \w \Vert_{4}	\leq	C_{K,d}\Vert	\Arm(\w)	\Vert_4,\quad \text{	for all } \;	\w\in	\mathbb{W}_0^{1,4}(\mathfrak{D}).
	\end{align}		
\item In view of \eqref{Sobolev-embedding3} and \eqref{Korn-ineq}, there exists a constant $\mathrm{M}_d>0$ such that
\begin{align}\label{Sobolev+Korn}
	\|\w\|_{\infty} \leq \mathrm{M}_d \|\Arm\w\|_{4}, \quad \text{	for all } \;	\w\in	\mathbb{W}_0^{1,4}(\mathfrak{D}).
\end{align}
	\end{itemize}

\end{lemma}

For the sake of simplicity, we do not distinguish between scalar, vector or matrix-valued   notations when it is clear from the context. In particular, $\Vert \cdot \Vert_E$  should be understood as follows
\begin{itemize}
	\item $\Vert f\Vert_E^2= \Vert f_1\Vert_E^2+\cdots+\Vert f_d\Vert_E^2$ for any $f=(f_1,\cdots,f_d) \in (E
	)^d$.
	\item $\Vert f\Vert_{E}^2= \displaystyle\sum_{i,j=1}^d\Vert f_{ij}\Vert_E^2$ for any $f\in \mathcal{M}_{d\times d}(E)$.
\end{itemize}

Throughout the article,  we denote by $C$   generic constant, which may vary from line to line.

\subsection{The Helmholtz projection}
The following content is motivated by the studies \cite{Farwig+Kozono+Sohr_2005,Farwig+Kozono+Sohr_2007,Kunstmann_2010}, which address the Helmholtz projection in the setting of general unbounded domains. We define
\begin{align}
	\widetilde{\L}^p(\mathfrak{D}) : = \begin{cases}
		\L^p(\mathfrak{D})\cap \L^2(\mathfrak{D}), & p\in[2,\infty)\\
		\L^p(\mathfrak{D})+ \L^2(\mathfrak{D}), & p\in (1,2),
	\end{cases}
\end{align}
and 
\begin{align}
	\widetilde{\L}^p_{\sigma}  : = \begin{cases}
		\L^p_{\sigma} \cap \L^2_{\sigma} , & p\in[2,\infty)\\
		\L^p_{\sigma} + \L^2_{\sigma}, & p\in (1,2).
	\end{cases}
\end{align}

Now, we recall that the Helmholtz projector $\mathcal{P}_{p}: \widetilde{\L}^p(\mathfrak{D}) \to \widetilde{\L}^p_{\sigma}$, which is a linear bounded operator characterized by the following decomposition
$\v=\mathcal{P}_p\v+\nabla \varphi, \;  \varphi \in \widetilde{\mathrm{G}}^p(\mathfrak{D})$ satisfying $(\mathcal{P}_p)^* = \mathcal{P}_{p'}$ with $\frac{1}{p}+\frac{1}{p'}=1$ (cf. \cite{Farwig+Kozono+Sohr_2007}), where
\begin{align}
	\widetilde{\mathrm{G}}^p(\mathfrak{D}) : = \begin{cases}
		\mathrm{G}^p(\mathfrak{D})\cap \mathrm{G}^2(\mathfrak{D}), & p\in[2,\infty)\\
		\mathrm{G}^p(\mathfrak{D})+ \mathrm{G}^2(\mathfrak{D}), & p\in (1,2),
	\end{cases}
\end{align}
and 
\begin{align}
	\mathrm{G}^p(\mathfrak{D}) = \{\nabla q \in \L^p(\mathfrak{D}) \; : \; q\in \mathrm{L}^p_{loc} (\mathfrak{D})\}.
\end{align}
Let us fix the notations $\widetilde{\mathbb{W}}^{1,p}_{0,\sigma}  := \widetilde{\mathbb{W}}_0^{1,p}(\mathfrak{D}) \cap \widetilde{\L}^{p}_{\sigma}$,  and $\widetilde{\mathbb{W}}^{-1,p}(\mathfrak{D}):= (\widetilde{\mathbb{W}}_0^{1,p'}(\mathfrak{D}))'$ and $\widetilde{\mathbb{W}}^{-1,p}_{\sigma} : = (\widetilde{\mathbb{W}}^{1,p'}_{0,\sigma})'$ with $\frac{1}{p}+\frac{1}{p'}=1$ for the dual spaces, where
\begin{align}
	\widetilde{\mathbb{W}}^{1,p}_0(\mathfrak{D}) : = \begin{cases}
		\mathbb{W}^{1,p}_0(\mathfrak{D})\cap \mathbb{W}^{1,2}_0(\mathfrak{D}), & p\in[2,\infty)\\
		\mathbb{W}^{1,p}_0(\mathfrak{D})+ \mathbb{W}^{1,2}_0(\mathfrak{D}), & p\in (1,2).
	\end{cases}
\end{align}

\begin{lemma}[{\cite[Proposition 3.1]{Kunstmann_2010}}]
	The Helmholtz projection $\mathcal{P}_{p}$ has a continuous linear extension $\widetilde{\mathcal{P}}_{p}: \widetilde{\mathbb{W}}^{-1,p}(\mathfrak{D}) \to \widetilde{\mathbb{W}}^{-1,p}_{\sigma}$ given by the restriction
	\begin{align}
		\widetilde{\mathcal{P}}_{p} \Phi  = \Phi |_{\widetilde{\mathbb{W}}^{1,p'}_{0,\sigma}}, \;\; \text{ for } \; \; \Phi \in \widetilde{\mathbb{W}}^{-1,p}_{\sigma} \;\; \text{ with } \; \;\frac{1}{p}+\frac{1}{p'}=1.
	\end{align}
\end{lemma}

\begin{remark}
	We will denote $\widetilde{\mathcal{P}}_{\frac43}$ by $\mathcal{P}$ in the sequel.
\end{remark}

\subsection{Linear and nonlinear operators}
Let us define linear operator $\A:\X\to \X'$ by $\A\v:= - \mathcal{P}\Delta\v$ such that
\begin{equation*}
	\langle \A\v, \u\rangle = (\nabla\v,\nabla\u).
\end{equation*}
Remember that the operator $\A$ is a non-negative, self-adjoint operator in $\H$  and \begin{align*}
	\left<\A\v,\v\right>=\|\v\|_{\V}^2,\ \textrm{ for all }\ \v\in\X, \ \text{ so that }\ \|\A\v\|_{\X^{\prime}}\leq \|\v\|_{\X}.
\end{align*}

\medskip

Next, we define the {trilinear form} $b(\cdot,\cdot,\cdot):\X\times\X\times\X\to\R$ by $$b(\u,\v,\w)=\int_{\mathfrak{D}}(\u(x)\cdot\nabla)\v(x)\cdot\w(x)\d x= \sum_{i,j=1}^d\int_{\mathfrak{D}}\u_i(x)\frac{\partial \v_j(x)}{\partial x_i}\w_j(x)\d x.$$ We also define an operator $\B: \X\times\X \to \X'$  by $\B(\u, \v):= \mathcal{P}(\u\cdot\nabla)\v$ such that
\begin{align*}
	\left\langle   \B(\u, \v), \w  \right \rangle = b(\u,\v,\w).
\end{align*}
 Using an integration by parts, it is immediate that 
\begin{equation}\label{b0}
	\left\{
	\begin{aligned}
		b(\u,\v,\v) &= 0, && \text{ for all }\ \u,\v \in\X,\\
		b(\u,\v,\w) &=  -b(\u,\w,\v), && \text{ for all }\ \u,\v,\w\in \X.
	\end{aligned}
	\right.\end{equation}
\begin{remark}\label{rem-trilinear-ext}
	For $\u,\v\in  \mathbb{L}^{4}_{\sigma}$ and $\w\in\mathbb{W}_0^{1,4}(\mathfrak{D})$, we have 
	\begin{align*}
		|b(\u,\v,\w)|=|b(\u,\w,\v)| \leq \|\u\|_{4}\|\v\|_{4}\|\nabla\w\|_{4} \leq C \|\u\|_{4}\|\v\|_{4}\|\w\|_{\mathbb{W}^{1,4}},
	\end{align*}
	for some positive constant $C$. Thus $b$ can be uniquely extended to the trilinear form (still denoted by the same) $b:\mathbb{L}^{4}_{\sigma}\times\mathbb{L}^{4}_{\sigma} \times \mathbb{W}_0^{1,4}(\mathfrak{D}) \to \R$. In parallel, the operator $\B$ can be uniquely extended to a bounded linear operator (still denoted by the same) $\B:\mathbb{L}^{4}_{\sigma}\times\mathbb{L}^{4}_{\sigma}\to \mathbb{W}^{-1,\frac43}(\mathfrak{D})$ such that 
	\begin{align*}
		\|\B(\u,\v)\|_{\mathbb{W}^{-1,\frac43}}\leq C \|\u\|_{4}\|\v\|_{4}, \;\;\; \u,\v\in\H.
	\end{align*}
\end{remark}

\medskip

Let us now define an operator $\J: \X \to \X'$  by $\J(\v):=-\mathcal{P}\diver(\Arm(\v)\Arm(\v))$ such that 
\begin{align*}
	\langle \J(\v),\u \rangle = \frac{1}{2}\int_{\mathfrak{D}} \Arm(\v(x))\Arm(\v(x)):\Arm(\u(x))\d x.
\end{align*}
 Note that for $\v\in\X$, we have 
\begin{align*}
	\|\J(\v)\|_{\X^{\prime}} \leq C \|\v\|^2_{\X}.
\end{align*}

\begin{remark}\label{J=0_d=2}
	In case of $\mathfrak{D}\subseteq\R^2$, due to divergence-free condition, we obtain $\Tr([\Arm(\v)]^3)=0$. Hence, for $\mathfrak{D}\subseteq\R^2$, we have for $\v\in\X$
	\begin{align*}
	\langle \J(\v),\v \rangle = 	\int_{\mathfrak{D}}\Tr[(\Arm(\v(x)))^3]\d x =0.
	\end{align*}
\end{remark}

\medskip

Finally, we define an operator $\K: \X \to \X'$  by $\K(\v):=-\mathcal{P}\diver(|\Arm(\v)|^{2}\Arm(\v))$ such that 
\begin{align*}
	\langle \K(\v),\u \rangle = \frac{1}{2}\int_{\mathfrak{D}} |\Arm(\v(x))|^2\Arm(\v(x)):\Arm(\u(x)) \d x.
\end{align*} It is immediate that $$\langle\mathcal{K}(\v),\v\rangle =\frac12\|\Arm(\v)\|_{4}^{4}.$$ Note that for $\v\in\X$, we have 
\begin{align*}
	\|{\K}(\v)\|_{\X^{\prime}} \leq C \|\v\|^3_{\X},
\end{align*}

\begin{lemma}[{\cite[Lemma 2.2]{Kinra+Cipriano_TGF}}]\label{lem-Loc-Monotinicity}
	Assume that  $\varepsilon_0:=1-\sqrt{\frac{\alpha^2}{2\beta\nu}}\in(0,1)$. Then, the operator $\mathscr{G}:\X\to\X^{\prime}$ given by 
	\begin{align}\label{opr-G}
		\mathscr{G}(\cdot):=\nu\A\cdot+\B(\cdot)+\alpha\J(\cdot)+\beta\K(\cdot)
	\end{align}
	satisfies
	\begin{align}\label{CL1}
		&	\left< \mathscr{G} (\v_1) - \mathscr{G} (\v_2), \v_1-\v_2   \right> + \frac{(\mathrm{M}_d)^2}{4\nu\varepsilon_0} \|\Arm(\v_2)\|_{{4}}^2 \|\v_1-\v_2\|_{2}^2 
		  \geq \frac{\nu\varepsilon_0}{4} \|\Arm(\v_1-\v_2)\|^2_{2} +  \frac{\beta\varepsilon_0}{4}\|\Arm(\v_1-\v_2)\|^4_{4} \geq 0,
	\end{align}
	for any  $\v_1,\v_2\in\X.$
\end{lemma}

\subsection{Abstract formulation and weak  solution}  Let us write the abstract formulation to system \eqref{equationV1} by taking the projection $\mathcal{P}$ as: 
\begin{equation}\label{TGF}
	\left\{
	\begin{aligned}
		\frac{\d\v(t)}{\d t}+\nu\A\v(t)+\B(\v(t))+\alpha\J(\v(t))+\beta\K(\v(t))& = \mathcal{P}\f(t), \ t\in (0,\infty),\\
		\v(0)&=\boldsymbol{\psi}\in\H,
	\end{aligned}
	\right.
\end{equation}
where $\f\in \H^{-1}(\mathfrak{D})$.  Let us now provide the notion	of weak	solution to system \eqref{TGF}.	
\begin{definition}\label{def-WS-TGF}
	A function $\v(\cdot)$ is called a \textit{weak solution} to system \eqref{equationV1} on time interval $[0, \infty)$, if $$\v \in  \mathrm{C}([0,\infty); \H) \cap \mathrm{L}^{2}_{\mathrm{loc}}(0,\infty; \V)\cap\mathrm{L}^{4}_{\mathrm{loc}}(0,\infty; \mathbb{W}^{1,4}_0(\mathfrak{D})),$$   $$\frac{\d\v}{\d t}\in\mathrm{L}^{2}_{\mathrm{loc}}(0,\infty;\V^{\prime})+\mathrm{L}^{\frac43}_{\mathrm{loc}}(0,\infty;\mathbb{W}^{-1,\frac43}(\mathfrak{D})),$$ and it satisfies 
	\begin{itemize}
		\item [(i)] for any $\psi\in \X,$ 
		\begin{align*}
			\left<\frac{\d\v(t)}{\d t}, \psi\right>&=  - \left\langle \nu \A\v(t)+\B(\v(t)) + \alpha \J(\v(t))+\beta \K(\v(t)) -\f , \psi \right\rangle,
		\end{align*}
		for a.e. $t\in[0,\infty);$
		\item [(ii)] the initial data:
		$$\v(0)=\boldsymbol{\psi} \ \text{ in }\ \H.$$
	\end{itemize}
\end{definition}

Next theorem  provides the existence and uniqueness of global weak solutions.

\begin{theorem}\label{solution}
		Let the condition \eqref{third-grade-paremeters-res} hold. For $\boldsymbol{\psi} \in \H$ and $\f\in \H^{-1}(\mathfrak{D})$, there exists a unique weak solution $\v(\cdot)$ to system \eqref{equationV1} in the sense of Definition \ref{def-WS-TGF}.
	\end{theorem}

	\section{Existence of singleton attractor with small forcing intensity} \label{sec3}\setcounter{equation}{0}
	
	It is evident from the works \cite{Kinra+Cipriano_TGF,Kinra+Cipriano_STGF} that the system \eqref{TGF} possesses a unique global attractor. In this section, we establish that under small forcing intensity  the attractor is a singleton set.	The following result is presented in light of the work  \cite{Kinra+Cipriano_TGF} (see \cite{Kinra+Cipriano_STGF} also), which establishes the existence of global attractors for system \eqref{TGF}.
	\begin{theorem}\label{thm-global-attractor}
		Assume that condition \eqref{third-grade-paremeters-res} is fulfilled and that Hypothesis \ref{assumpO} holds. Let $\mathfrak{D}$ denote either a bounded or an unbounded domain, and suppose that $\f \in \H^{-1}(\mathfrak{D})$ in the former case or $\f \in \L^{2}(\mathfrak{D})$ in the latter. Then, the continuous semigroup $\S(\cdot)$ associated with system \eqref{TGF} admits a unique global attractor $\mathscr{A}$.
	\end{theorem}
	%
	%
	%

In order to show that the global attractor $\mathscr{A}$ has singleton structure, we need the following smallness assumption on the external forcing $\f$:

\begin{hypothesis}\label{hypo-small-force}
	We assume that $\f$ satisfies the following condition:
	\begin{align}\label{C_2}
		\varrho:=\nu\varepsilon_0\lambda  -  \frac{(\mathrm{M}_d)^2}{2\nu\varepsilon_0} \left( \frac{1}{\beta^*\nu^{*}}\right)^{\frac12} \left( \frac{1}{2} + \frac{1}{\nu^*\lambda} \right)^{\frac12} \|\f\|_{\H^{-1}} >0,
	\end{align}
	where $\mathrm{M}_d$ is the constant appearing in \eqref{Sobolev+Korn}.
\end{hypothesis}

The following lemma tells us about some useful estimated satisfied by the weak solution of system \eqref{equationV1} which will be used in the sequel.

\begin{lemma}\label{lem-Absorb}
Assume that $\f\in \H^{-1}(\mathfrak{D})$, the condition \eqref{third-grade-paremeters-res} is satisfied and the Hypothesis \ref{assumpO} holds. Then, for any bounded set $B_{\H}\subset \H$  and for $\wp\in(0,   \nu^*\lambda]$, there exists a time $\mathcal{T}_{\wp}>0$ such that any weak solution $\v(\cdot)$ of system \eqref{TGF} with initial data in $B_{\H}$ satisfies
	\begin{align}\label{absorb1}
	\|\v(t)\|^2_{2} + \beta^* \int_{0}^{t} e^{-\wp (t- \xi ) } \|\Arm(\v(\xi))\|^4_{4} \d s &\leq  \frac{\|\f\|_{\H^{-1}}^2}{\nu^*\wp}, \;\; \; \text{	for all }\; t\geq \mathcal{T}_{\wp},
\end{align}
and 
\begin{align}\label{absorb2}
	\int_{s}^{t} \|\Arm(\v(\xi))\|^2_{4} \d \xi \leq  \left( \frac{1}{\beta^*\nu^{*}}\right)^{\frac12} \left( \frac{1}{2} + \frac{1}{\nu^*\lambda} \right)^{\frac12} \|\f\|_{\H^{-1}} (t-s) ,\;\; \; \text{	for all } s\geq \mathcal{T}_{ \nu^*\lambda} \; \text{ and } \; t> s+1,
\end{align}
for any $\boldsymbol{\psi}\in B_{\H}$, where 
\begin{align}\label{nu_ast}
	\nu^{\ast}:=\begin{cases}
		\nu, & for \;\; d=2;\\
		\frac{\nu}{2}\left(1+\varepsilon_0\right), & for \;\; d=3,
	\end{cases} 
	\;\;\; \text{ and } \;\;\; \beta^{\ast}:=\begin{cases}
		\beta, & for \;\; d=2;\\
		\beta\varepsilon_0, & for \;\; d=3.
	\end{cases} 
\end{align}
\end{lemma}

\begin{proof}
	By taking the inner product with $\v(\cdot)$ to the equation $\eqref{TGF}_1$, and using $\langle \B(\v),\v\rangle =0$, we find
	\begin{align}\label{Sing-1}
		\frac{1}{2}\frac{\d}{\d \xi}\|\v(\xi)\|^2_{2}
		&=-\frac{\nu}{2}\|\Arm(\v(\xi))\|^2_{2} - \frac{\beta}{2}\|\Arm(\v(\xi))\|^4_{4} + \frac{\alpha}{2} \int_{\mathfrak{D}}(\Arm(\v(x,\xi)))^2:\Arm(\v(x,\xi))\d x +\langle\f(\xi),\v(\xi)\rangle,
	\end{align}
	for a.e. $\xi\in[0,T]$ with $T>0$. Making use of the H\"older's inequality, \eqref{third-grade-paremeters-res} and Young's inequality, we obtain  
	\begin{align}\label{Sing-2}
		\left\vert\frac{\alpha}{2} \int_{\mathfrak{D}}(\Arm(\v))^2:\Arm(\v)\d x \right\vert & =	\left\vert\frac{\alpha}{2} \int_{\mathfrak{D}}\Tr[(\Arm(\v))^3]\d x \right\vert
		\leq  
		\begin{cases}
			0, & d=2;\\
			\frac{\nu(1-\varepsilon_0)}{4} \|\Arm(\v)\|^2_{2}  + \frac{\beta(1-\varepsilon_0)}{2}  \|\Arm(\v)\|^4_{4},&  d=3,
		\end{cases}
	\end{align}
	where we have also used Remark \ref{J=0_d=2} for $d=2$. Using the Cauchy-Schwarz inequality, \eqref{Poin}, \eqref{Korn-ineq} and Young's inequality, we get 
	\begin{align}\label{Sing-3}
		|\langle\f,\v\rangle|
		&\leq \|\f\|_{\H^{-1}}\|\v\|_{\V} 
	 \leq \frac12\|\f\|_{\H^{-1}}\|\Arm(\v)\|_{2}
	 \leq \frac{\nu^{\ast}}{4}\|\Arm(\v)\|_{2}^2  + \frac{\|\f\|_{\H^{-1}}^2}{4\nu^*}, 
	\end{align}
where $\nu^*>0$ is given by \eqref{nu_ast}.	Combining \eqref{Sing-1}-\eqref{Sing-3}, we obtain	
	\begin{align}\label{3p8}
		&\frac{\d}{\d \xi}\|\v(\xi)\|^2_{2} +\frac{\nu^{\ast}}{2}\|\Arm(\v(\xi))\|^2_{2} + \beta^*\|\Arm(\v(\xi))\|^4_{4}
		\leq \frac{\|\f\|_{\H^{-1}}^2}{2\nu^*},
	\end{align}
	for a.e. $\xi\in[0,T]$ with $T>0$, where  $\nu^*, \beta^*>0$ are given by \eqref{nu_ast},  which due to Poincar\'e inequality \eqref{Poin} implies
	\begin{align}\label{Sing-4}
		&\frac{\d}{\d\xi}\|\v(\xi)\|^2_{2} + \nu^{\ast}\lambda \|\v(\xi)\|^2_{2}   + \beta^*\|\Arm(\v(\xi))\|^4_{4}
		\leq \frac{\|\f\|_{\H^{-1}}^2}{2\nu^*} ,
	\end{align}
	for a.e. $\xi\in[0, T]$ with $T>0$, where $\nu^{\ast}>0$ is a constant defined in \eqref{nu_ast}. Using the variation of constants formula, for any $0< \wp \leq \nu^*\lambda$, we write 
	\begin{align}\label{Sing-5}
		e^{\wp t}\|\v(t)\|^2_{2} + \beta^* \int_{0}^{t} e^{\wp s}\|\Arm(\v(s))\|^4_{4} \d s &\leq  \|\boldsymbol{\psi}\|^2_{2} +  \frac{\|\f\|_{\H^{-1}}^2}{2\nu^*} \int_{0}^{t}e^{\wp \zeta} \d\zeta = \|\boldsymbol{\psi}\|^2_{2} +  \frac{\|\f\|_{\H^{-1}}^2}{2\nu^*\wp } (e^{\wp t} -1 ),
	\end{align}
which implies
\begin{align*}
	\|\v(t)\|^2_{2}  + \beta^* \int_{0}^{t} e^{-\wp (t- \xi ) } \|\Arm(\v(\xi))\|^4_{4} \d \xi &\leq  \|\boldsymbol{\psi}\|^2_{2}e^{-\wp t} +  \frac{\|\f\|_{\H^{-1}}^2}{2\nu^*\wp},
\end{align*}
	for all $t>0$. There exists a time $\mathcal{T}_{\wp}>0$  such that  
	\begin{align}\label{Sing-6}
		\|\v(t)\|^2_{2}  + \beta^* \int_{0}^{t} e^{-\wp (t- \xi ) } \|\Arm(\v(\xi))\|^4_{4} \d \xi   &\leq  \frac{\|\f\|_{\H^{-1}}^2}{\nu^*\wp }, \;\; \; \text{	for all } t\geq \mathcal{T}_{\wp}.
	\end{align}
Now, integrating \eqref{3p8} from $s\geq \mathcal{T}_{\nu^*\lambda}$ to $t \geq s+1$ and using \eqref{Sing-6}, we get
	\begin{align}
	 \beta^*\int_{s}^{t} \|\Arm(\v(\xi))\|^4_{4} \d \xi
	& \leq  \|\v(s)\|^2_{2}+  \frac{\|\f\|_{\H^{-1}}^2}{2\nu^*} (t-s)
	  \leq \frac{\|\f\|_{\H^{-1}}^2}{(\nu^*)^2\lambda} +  \frac{\|\f\|_{\H^{-1}}^2}{2\nu^*} (t-s) ,  \nonumber
	 \\ & \leq \left[\frac{1}{(\nu^*)^2\lambda} +  \frac{1}{2\nu^*}\right] \|\f\|_{\H^{-1}}^2 (t-s),  
\end{align}
which gives
\begin{align*}
\int_{s}^{t} \|\Arm(\v(\xi))\|^2_{4} \d \xi \leq 	(t-s)^{\frac{1}{2}}\left(\int_{s}^{t} \|\Arm(\v(\xi))\|^4_{4} \d \xi \right)^{\frac{1}{2}}
	& \leq  \left( \frac{1}{\beta^*}\right)^{\frac12} \left[\frac{1}{(\nu^*)^2\lambda} +  \frac{1}{2\nu^*}\right]^{\frac12} \|\f\|_{\H^{-1}} (t-s) ,
\end{align*}
for all $s\geq \mathcal{T}_{\nu^*\lambda}$ and $t \geq s+1$.	This completes the proof.
\end{proof}

Next result is crucial to obtain the singleton structure of global attractor $\mathscr{A}$.
\begin{lemma}\label{Uniqueness}
	Assume that $\f\in \H^{-1}(\mathfrak{D})$, the condition \eqref{third-grade-paremeters-res} is satisfied and the Hypothesis \ref{assumpO} holds. For any bounded set $B_{\H}\subset \H$ and any two weak solutions $\v_1(\cdot)$ and $\v_2(\cdot)$ of system \eqref{TGF} in the sense of Definition \ref{def-WS-TGF} with initial data in $B_{\H}$ satisfy
	\begin{align}\label{Unique}
		\|\v_1(t_2)-\v_2(t_2)\|^2_{2} \leq  \|\v_1(t_1)-\v_2(t_1)\|^2_{2}\exp\left\{ -\nu\varepsilon_0\lambda(t_2-t_1)  + \frac{(\mathrm{M}_d)^2}{2\nu\varepsilon_0} \int_{t_1}^{t_2}\|\Arm(\v_1(\xi))\|^2_{4}\d \xi \right\},
	\end{align}
for any $t_2> t_1 \geq0$, where $\mathrm{M}_d>0$ is a constant appearing in \eqref{Sobolev+Korn}.
\end{lemma}
\begin{proof}
	Since $\v_1(\cdot)$ and $\v_2(\cdot)$ are the solutions of \eqref{TGF}, therefore $\mathfrak{Y}(\cdot):=\v_1(\cdot)-\v_2(\cdot)$ satisfies the following:
	\begin{align}\label{uni1}
		\frac{\d \mathfrak{Y}(t)}{\d t} &= -\left[\mathscr{G} (\v_1(t))-\mathscr{G} (\v_2(t))\right],
	\end{align}
	for a.e. $t\in[0,T]$, where $\mathscr{G}(\cdot)$ is given by \eqref{opr-G}. 	Taking the inner product of \eqref{uni1} with $\mathfrak{Y}(\cdot)$ and using \eqref{CL1}, we get
	\begin{align}\label{uni2}
		&	\frac{1}{2}\frac{\d}{\d t} \|\mathfrak{Y}(t)\|^2_{2}  +  \frac{\nu\varepsilon_0}{4} \|\Arm(\mathfrak{Y}(t))\|^2_{2} +  \frac{\beta\varepsilon_0}{4}\|\Arm(\mathfrak{Y}(t))\|^4_{4}   \leq   \frac{(\mathrm{M}_d)^2}{4\nu\varepsilon_0} \|\Arm(\v_2(t))\|_{{4}}^2 \|\mathfrak{Y}(t)\|_{2}^2  ,
	\end{align}
for a.e. $t\in[0,T]$. 	Using \eqref{Poin} in \eqref{uni2}, we write
	\begin{align*}
		\frac{\d}{\d t}\|\mathfrak{Y}(t)\|^2_{2} + \left(\nu\varepsilon_0\lambda - \frac{(\mathrm{M}_d)^2}{2\nu\varepsilon_0} \|\Arm(\v_1(t))\|^2_{4}\right)\|\mathfrak{Y}(t)\|^2_{2}\leq 0.
	\end{align*}
Hence,	an application of Gronwall's inequality concludes the proof.
\end{proof}

Now we are ready to establish the main result of this section, that is, the global attractor has only one element.

\begin{theorem}\label{D-SA}
	Assume that $\f\in \H^{-1}(\mathfrak{D})$, the condition \eqref{third-grade-paremeters-res} is satisfied and the Hypotheses \ref{assumpO} and \ref{hypo-small-force} hold.  Let $\mathfrak{D}$ denote either a bounded or an unbounded domain, and suppose that $\f \in \H^{-1}(\mathfrak{D})$ in the former case or $\f \in \L^{2}(\mathfrak{D})$ in the latter.  Then, the global attractor $\mathscr{A}$ of system \eqref{TGF} obtained in Theorem \ref{thm-global-attractor} is a singleton set $\mathscr{A}=\{\textbf{a}_{*}\} \subset \H$.
\end{theorem}
\begin{proof}
	Let us take two arbitrary points $\textbf{a}_1, \textbf{a}_2\in \mathscr{A}$. Since the global attractor $\mathscr{A}$ is a bounded set in $\H$ and consists of bounded complete trajectories in $\H$, there exists two complete trajectories $\boldsymbol{\phi}_1(t)$ and $\boldsymbol{\phi}_2(t)$ with $\boldsymbol{\phi}_1(0)=\textbf{a}_1$ and $\boldsymbol{\phi}_2(0)=\textbf{a}_2$, respectively. Hence, in view of Lemmas \ref{lem-Absorb} and \ref{Uniqueness}, and small forcing intensity condition \eqref{C_2}, we observe that
	\begin{align*}
	&	\|\textbf{a}_1-\textbf{a}_2\|^2_{2} 
	\nonumber\\ &= \|\boldsymbol{\phi}_1(0)-\boldsymbol{\phi}_2(0)\|^2_{2}\\
		&=\|\v_1(t,\boldsymbol{\phi}_1(-t))-\v_2(t,\boldsymbol{\phi}_2(-t))\|^2_{2}  \\
		&\leq \|\v_1(\mathcal{T}_{\nu^*\lambda},\boldsymbol{\phi}_1(-t)) - \v_2(\mathcal{T}_{\nu^*\lambda},\boldsymbol{\phi}_2(-t))\|^2_{2} \nonumber\\ & \quad \times \exp\left\{ -\nu\varepsilon_0\lambda(t-\mathcal{T}_{\nu^*\lambda})  + \frac{(\mathrm{M}_d)^2}{2\nu\varepsilon_0} \int_{\mathcal{T}_{\nu^*\lambda}}^{t} \|\Arm(\v_1(\xi))\|^2_{4}\d \xi \right\} ,  && \text{ for all }\  t> \mathcal{T}_{\nu^*\lambda}+1,\\
		&\leq 2 \left( \|\v_1(\mathcal{T}_{\nu^*\lambda},\boldsymbol{\phi}_1(-t))\|_2^2 + \|\v_2(\mathcal{T}_{\nu^*\lambda},\boldsymbol{\phi}_2(-t))\|^2_{2}\right)  
		\nonumber\\ & \quad \times \exp\left\{ -\nu\varepsilon_0\lambda(t-\mathcal{T}_{\nu^*\lambda})  + \frac{(\mathrm{M}_d)^2}{2\nu\varepsilon_0} \left( \frac{1}{\beta^*\nu^{*}}\right)^{\frac12} \left( \frac{1}{2} + \frac{1}{\nu^*\lambda} \right)^{\frac12} \|\f\|_{\H^{-1}} (t-\mathcal{T}_{\nu^*\lambda}) \right\} , && \text{ for all }\  t\geq  \mathcal{T}_{\nu^*\lambda}+1,\\
		& \leq  \frac{4\|\f\|_{\H^{-1}}^2}{(\nu^*)^2\lambda}  \exp\left\{ -\varrho  (t-\mathcal{T}_{\nu^*\lambda}) \right\} , && \text{ for all }\  t\geq  \mathcal{T}_{\nu^*\lambda}+1,\\
		& \to 0 \ \text{ as } \ t\to \infty,
	\end{align*}
where we have used \eqref{C_2} to obtain the final convergence, which concludes the proof.
\end{proof}

	\section{Stochastic counterpart of underlying system with additive noise} \label{sec4}\setcounter{equation}{0}

	Let us consider the stochastic version of system \eqref{TGF} with infinite-dimensional additive white noise which reads as follows:
	\begin{equation}\label{STGF}
		\left\{
		\begin{aligned}
			\d\v_{\varsigma}(t)+\{\nu\A\v_{\varsigma}(t)+\B(\v_{\varsigma}(t))+\alpha\J(\v_{\varsigma}(t))+\beta\K(\v_{\varsigma}(t))\}\d t&=\mathcal{P}\f \d t + \varsigma \d\mathrm{W}(t), \ \ \ t> 0, \\ 
			\v(0)&=\boldsymbol{\psi},
		\end{aligned}
		\right.
	\end{equation}
	where $\boldsymbol{\psi}\in \H,\ \f\in \H^{-1}(\mathfrak{D})$ and $\{\mathrm{W}(t)\}_{t\in \R}$ is a two-sided cylindrical Wiener process in $\H$ with its RKHS $\mathrm{K}$ which satisfies the following hypothesis:
	\begin{hypothesis}\label{assump1}
		$ \mathrm{K} \subset \H \cap {\mathbb{W}}^{1,4}(\mathfrak{D}) $ is a Hilbert space such that for some $\delta\in (0, 1/2),$
		\begin{align}\label{A1}
			\A^{-\delta} : \mathrm{K} \to \H \cap {\mathbb{W}}^{1,4}(\mathfrak{D})  \  \text{ is }\ \gamma \text{-radonifying.}
		\end{align}
	\end{hypothesis}

	We now introduce the notion of a weak solution (in the analytic sense) to system \eqref{STGF}, corresponding to the initial condition $\x\in\H$ prescribed at the initial time $s\in \R.$
	
	\begin{definition}\label{Def_u}
		Suppose that the condition \eqref{third-grade-paremeters-res} holds, and Hypotheses \ref{assump1} and \ref{assumpO} are satisfied. If $\x\in \H$, $s\in \R$, $\f\in \H^{-1}(\mathfrak{D})$ and $\{\W(t)\}_{t\in \R}$ is a two-sided Wiener process with its RKHS $\mathrm{K}$. A process $\{\v(t), \ t\geq s\},$ with trajectories in $\mathrm{C}([s, \infty); \H) \cap \mathrm{L}^{4}_{\mathrm{loc}}([s, \infty); \mathbb{W}^{1,4}(\mathfrak{D}))$ is a solution to system \eqref{STGF} if and only if  $\v_{\varsigma}(s) = \x$ and for any $\boldsymbol{\phi}\in \X$, $t>s,$ $\mathbb{P}$-a.s.,
		\begin{align*}
			&	(\v_{\varsigma}(t), \boldsymbol{\phi}) - (\v_{\varsigma}(s), \boldsymbol{\phi}) 
			 =  - \int_{s}^{t} \langle \nu\A\v_{\varsigma}(\xi)+\B(\v_{\varsigma}(\xi))+\alpha\J(\v_{\varsigma}(\xi))+\beta\K(\v_{\varsigma}(\xi))  - \f , \boldsymbol{\phi} \rangle \d \xi  + \varsigma \int_{s}^{t} ( \boldsymbol{\phi}, \d\W(\xi)).
		\end{align*}
	\end{definition}
	
	The well-posedness of   system \eqref{STGF} is discussed in Theorem \ref{STGF-Sol} below.


	\subsection{Metric dynamical system}Let us denote $\mathfrak{X} = \H \cap  {\mathbb{W}}^{1,4}(\mathfrak{D}) $. Let $\mathrm{E}$ denote the completion of $\A^{-\delta}(\mathfrak{X})$ with respect to the graph norm $\|x\|_{\mathrm{E}}=\|\A^{-\delta} x \|_{\mathfrak{X}}, \text{ for } x\in \mathfrak{X}$,  where $ \|\cdot\|_{\mathfrak{X}} = \|\cdot\|_{2} +  \|\cdot\|_{\mathbb{W}^{1,4}}$. Note that $\mathrm{E}$ is a separable Banach spaces (cf. \cite{Brze2}).
	
	For $\xi \in(0, 1/2)$, let us set 
	$$ \|\omega\|_{C^{\xi}_{1/2} (\mathbb{R};\mathrm{E})} = \sup_{t\neq s \in \mathbb{R}} \frac{\|\omega(t) - \omega(s)\|_{\mathrm{E}}}{|t-s|^{\xi}(1+|t|+|s|)^{1/2}}.$$
	Furthermore, we define
	\begin{align*}
		C^{\xi}_{1/2} (\mathbb{R}; \mathrm{E}) &= \left\{ \omega \in C(\mathbb{R}; \mathrm{E}) : \omega(0)=\boldsymbol{0},\ \|\omega\|_{C^{\xi}_{1/2} (\mathbb{R}; \mathrm{E})} < \infty \right\},\\ \Omega(\xi, \mathrm{E})&= \overline{\{ \omega \in C^\infty_0 (\mathbb{R}; \mathrm{E}) : \omega(0) = 0 \}}^{C^{\xi}_{1/2} (\mathbb{R}; \mathrm{E})}.
	\end{align*}
	The space $\Omega(\xi, \mathrm{E})$ is a separable Banach space. We also define
	$$C_{1/2} (\mathbb{R}; \mathrm{E}) = \left\{ \omega \in C(\mathbb{R}; \mathrm{E}) : \omega(0)=0, \|\omega\|_{C_{1/2} (\mathbb{R}; \mathrm{E})} = \sup_{t \in \mathbb{R}} \frac{\|\omega(t) \|_{\mathrm{E}}}{1+|t|^{\frac{1}{2}}} < \infty \right\}.$$

	Let us denote by $\mathcal{F}$, the Borel $\sigma$-algebra on $\Omega(\xi, \mathrm{E}).$ For $\xi\in (0, 1/2)$, there exists a Borel probability measure $\mathbb{P}$ on $\Omega(\xi, \mathrm{E})$ (cf. \cite{Brze}) such that the canonical process $\{w_t, \ t\in \mathbb{R}\}$ defined by 
	\begin{align}\label{Wp}
		w_t(\omega) := \omega(t), \ \ \ \omega \in \Omega(\xi, \mathrm{E}),
	\end{align}
	is an $\mathrm{E}$-valued two-sided Wiener process such that the RKHS of the Gaussian measure $\mathscr{L}(w_1)$ on $\mathrm{E}$ is $\mathrm{K}$. For $t\in \mathbb{R},$ let $\mathcal{F}_t := \sigma \{ w_s : s \leq t \}.$ Then  there exists a unique bounded linear map $\mathrm{W}(t): \mathrm{K} \to \mathrm{L}^2(\Omega(\xi, \mathrm{E}), \mathcal{F}_t  ,  \mathbb{P}).$ Moreover, the family $\{\mathrm{W}(t)\}_{t\in \mathbb{R}}$ is a $\mathrm{K}$-cylindrical Wiener process on a filtered probability space $(\Omega(\xi, \mathrm{E}), \mathcal{F}, \{\mathcal{F}_t\}_{t \in \mathbb{R}} , \mathbb{P})$ (cf. \cite{BP} for more details).
	
	We consider a flow $\vartheta = (\vartheta_t)_{t\in \mathbb{R}}$ on the space $C_{1/2} (\mathbb{R}; \mathrm{E}),$  defined by
	$$ \vartheta_t \omega(\cdot) = \omega(\cdot + t) - \omega(t), \ \ \ \omega\in C_{1/2} (\mathbb{R};\mathrm{E}), \ \ t\in \mathbb{R}.$$ 
	This flow keeps the spaces $C^{\xi}_{1/2} (\mathbb{R};\mathrm{E})$ and $\Omega(\xi, \mathrm{E})$ invariant and preserves $\mathbb{P}.$ 
	
	Summing up, we have the following result:
	\begin{proposition}[{\cite[Proposition 6.13]{Brzezniak+Li_2006}}]\label{m-DS1}
		The quadruple $(\Omega(\xi, \mathrm{E}), \mathcal{F}, \mathbb{P}, {\vartheta})$ is a metric dynamical system. 
	\end{proposition}

	\subsection{Ornstein-Uhlenbeck process}\label{O-Up}
	Let us first recall some analytic preliminaries from \cite{Brzezniak+Li_2006} which will help us to define an Ornstein-Uhlenbeck process. All the results of this subsection are valid for the space $\mathrm{C}^{\xi}_{1/2} (\mathbb{R}; \mathbb{Y})$ replaced by $\Omega(\xi, \mathbb{Y}).$ 
	\begin{proposition}[{\cite[Proposition 2.11]{Brzezniak+Li_2006}}]\label{Ap}
		Let  $-\mathbb{A}$ be the generator of an analytic semigroup $\{e^{t\mathbb{A}}\}_{t\geq 0}$ on a separable Banach space $\mathbb{Y}$ such that for some $C>0\ \text{and}\ \gamma>0$
		\begin{align}\label{ASG}
			\| \mathbb{A}^{1+\delta}e^{-t\mathbb{A}}\|_{\mathfrak{L}(\mathbb{Y})} \leq C_{\delta} t^{-1-\delta} e^{-\gamma t}, \ \ t> 0,
		\end{align}
		where $\mathfrak{L}(\mathbb{Y})$ denotes the space of all bounded linear operators from $\mathbb{Y}$ to $\mathbb{Y}$.	For $\xi \in (\delta, 1/2)$ and $\widetilde{\omega} \in  \mathrm{C}^{\xi}_{1/2} (\mathbb{R};\mathbb{Y}),$  define 
		\begin{align}
			\hat{\z}(t) := \hat{\z} (\mathbb{A} ; \widetilde{\omega}, t) := \int_{-\infty}^{t} \mathbb{A}^{1+\delta} e^{-(t-r)\mathbb{A}} (\widetilde{\omega}(t) - \widetilde{\omega}(r))\d r, \ \ t\in \mathbb{R}.
		\end{align}
		If $t\in \mathbb{R},$ then $\hat{\z}(t)$ is a well-defined element of $\mathbb{Y}$ and the mapping 
		$$\mathrm{C}^{\xi}_{1/2} (\mathbb{R};\mathbb{Y}) \ni \widetilde{\omega}  \mapsto \hat{\z}(t) \in \mathbb{Y} $$
		is continuous. Moreover, the map $\hat{\z} :  \mathrm{C}^{\xi}_{1/2} (\mathbb{R}; \mathbb{Y}) \to  \mathrm{C}_{1/2} (\mathbb{R}; \mathbb{Y})$  is well defined, linear and bounded. In particular, there exists a constant $C >0$ such that for any $\widetilde{\omega} \in \mathrm{C}^{\xi}_{1/2} (\mathbb{R};\mathbb{Y})$ 
		\begin{align}\label{X_bound_of_z}
			\|\hat{\z}(\widetilde{\omega},t)\|_{\mathbb{Y}} \leq C(1 + |t|^{1/2})\|\widetilde{\omega}\|_{\mathrm{C}^{\xi}_{1/2} (\mathbb{R}; \mathbb{Y})}, \ \ \ t \in \R.
		\end{align}
		Furthermore, under the same assumption, the following results hold (Corollaries 6.4, 6.6 and 6.8 in \cite{Brzezniak+Li_2006}):
		\begin{itemize}
			\item [1.]For all $-\infty<a<b<\infty$ and $t\in \R$, the map 
			\begin{align}\label{O-U_conti}
				\mathrm{C}^{\xi}_{1/2} (\mathbb{R};\mathbb{Y}) \ni \widetilde{\omega} \mapsto (\hat{\z}(\widetilde{\omega}, t), \hat{\z}(\widetilde{\omega},\cdot)) \in \mathbb{Y} \times \mathrm{L}^{q} (a, b; \mathbb{Y}),
			\end{align}
			where $q\in [1, \infty]$, is continuous.
			\item [2.] For any $\omega \in \mathrm{C}^{\xi}_{1/2} (\mathbb{R};\mathbb{Y}),$
			\begin{align}\label{stationary}
				\hat{\z}(\vartheta_s \omega, t) = \hat{\z}(\omega, t+s), \ \ t, s \in \mathbb{R}.
			\end{align}
			\item [3.] For $\zeta \in \mathrm{C}_{1/2}(\mathbb{R};\mathbb{Y}),$ if we put $\uptau_s(\zeta(t))=\zeta(t+s), \ t,s \in \R,$ then, for $t \in \R ,\  \uptau_s \circ \hat{\z} = \hat{\z}\circ\vartheta_s$, that is, 
			\begin{align}\label{IS}
				\uptau_s\big(\hat{\z}(\omega, \cdot)\big)= \hat{\z}\big(\vartheta_s(\omega), \cdot\big), \ \ \ \omega\in \mathrm{C}^{\xi}_{1/2} (\mathbb{R};\mathbb{Y}).
			\end{align}
		\end{itemize} 
	\end{proposition}

	Next, we define the Ornstein-Uhlenbeck process under Hypothesis \ref{assump1}. For $\delta$ as in Hypothesis \ref{assump1}, $\nu> 0$, $ \xi \in (\delta, 1/2)$ and $ \omega \in C^{\xi}_{1/2} (\mathbb{R};\mathrm{E})$ (so that $(\nu \A + \I)^{-\delta}\omega \in C^{\xi}_{1/2} (\mathbb{R};\mathfrak{X})$), we define $$ \z (\omega, \cdot) := \hat{\z}((\nu \A + \I); (\nu \A +  \I)^{-\delta}\omega, \cdot) \ \in C_{1/2}(\mathbb{R};\mathfrak{X}),$$  that is, for any $t\geq 0,$ 
	\begin{align}\label{DOu1}
		\z (\omega, t)&=\int_{-\infty}^{t} (\nu \A +  \I)^{1+\delta} e^{-(t-\tau)(\nu \A +  \I)} ((\nu \A +  \I)^{-\delta}\vartheta_{\tau} \omega)(t-\tau)\d \tau.
	\end{align}
	For $\omega \in C^{\infty}_0 (\mathbb{R};\mathrm{E})$ with $\omega(0)= \boldsymbol{0},$ using  integration by parts, we obtain 
	\begin{align*}
		\frac{\d\z(t)}{\d t} &= -(\nu \A +  \I )\int_{-\infty}^{t} (\nu \A +  \I)^{1+\delta} e^{-(t-r)(\nu \A +  \I)} [(\nu \A +  \I)^{-\delta}\omega(t) \\&\qquad\qquad - (\nu \A +  \I)^{-\delta}\omega(r)]\d r +  \frac{\d\omega(t)}{\d t}.
	\end{align*}
	Thus $\z(\cdot)$ is the solution of the following equation:
	\begin{align}\label{OuE1}
		\frac{\d\z (t)}{\d t} + (\nu \A +  \I)\z(t) = \frac{\d\omega(t)}{\d t}, \ \ t\in \mathbb{R}.
	\end{align}

	According to the  definition \eqref{Wp} of Wiener process $\{w_t, \ t\in \R\},$ one can view the formula \eqref{DOu1} as a definition of a process $\{\z(t)\}_{ t\in \R}$ on the probability space $(\Omega(\xi, \mathrm{E}), \mathcal{F}, \mathbb{P})$. Equation \eqref{OuE1} clearly tells that the process $\z(\cdot)$ is an Ornstein-Uhlenbeck process. Furthermore, the following results hold for $\z(\cdot)$.

	\begin{proposition}[{\cite[Proposition 6.10]{Brzezniak+Li_2006}}]\label{SOUP1}
		The process $\{\z(t)\}_{ t\in \mathbb{R}},$ is a stationary Ornstein-Uhlenbeck process on $(\Omega(\xi, \mathrm{E}), \mathcal{F}, \mathbb{P})$. It is a solution of the equation 
		\begin{align}\label{OUPe1}
			\d\z(t) + (\nu \A +  \I)\z(t) \d t = \d\mathrm{W}(t), \ \ t\in \mathbb{R},
		\end{align}
		that is, for all $t\in \mathbb{R},$ $\mathbb{P}$-a.s.,
		\begin{align}\label{oup1}
			\z (t) =  \int_{-\infty}^{t} e^{-(t-\xi)(\nu \A + \I)} \d\mathrm{W}(\xi),
		\end{align}
		where the integral is an It\^o integral on the M-type 2 Banach space $\mathfrak{X}$  (cf. \cite{Brze1}).  
	\end{proposition}

	Let us now provide some consequences of the previous discussion which will be used in the sequel. 
	\begin{lemma}[{\cite[Lemma 3.5]{Kinra+Cipriano_STGF}}]\label{Bddns4}
		For each $\omega\in \Omega$ and $c>0$, we obtain 
		\begin{align*}
			\lim_{t\to - \infty} \|\z(\omega,t)\|^2_{2}\  e^{c t} = 0.
		\end{align*}
	\end{lemma}

	\begin{lemma}[{\cite[Lemma 3.6]{Kinra+Cipriano_STGF}}]\label{Bddns5}
		For each $\omega\in \Omega$ and $c>0$, we get 
		\begin{align*}
			\int_{- \infty}^{0} \bigg\{ 1 + \|\z(\omega,t)\|^2_{2} +  \|\z(\omega,t)\|^4_{\mathbb{W}^{1,4}} \bigg\}e^{c t} \d t < \infty.
		\end{align*}
	\end{lemma}

	\begin{definition}\label{RA2}
		A function $\kappa: \Omega\to (0, \infty)$ belongs to the class $\mathfrak{K}$ if and only if 
		\begin{align}
			\lim_{t\to \infty} [\kappa(\vartheta_{-t}\omega)]^2 e^{-c t } = 0, 
		\end{align}
		for all $c>0$.
	\end{definition}
	Let us denote the class of all closed and bounded random sets $D$ on $\H$ by $\mathfrak{DK}$ such that the radius function $\Omega\ni \omega \mapsto \kappa(D(\omega)):= \sup\{\|x\|_{2}:x\in D(\omega)\}$ belongs to the class $\mathfrak{K}.$

	\subsection{Random dynamical system}

	Remember that Hypothesis \ref{assump1} is satisfied and that $\delta$ has the property stated there. Let us fix $\nu> 0$, and the parameter $\xi \in (\delta, 1/2)$.
	
	Using a random transformation (known as Doss-Sussmann transformation, \cite{Doss_1977,Sussmann_1978}), we get a random partial differential equation which equivalent to   system \eqref{STGF}. Let us define 
	\begin{align}\label{D-S_trans}
		\y^{\varsigma}(\cdot):=\v_{\varsigma}(\cdot) - \varsigma\z(\omega,\cdot).
	\end{align}
	 Then $\y(\cdot)$ satisfies the following system:
	\begin{equation}\label{CTGF}
		\left\{
		\begin{aligned}
			\frac{\d\y^{\varsigma}}{\d t} &= -\nu \A\y^{\varsigma} -  \B(\y^{\varsigma} + {\varsigma} \z) -\alpha\J(\y^{\varsigma} + {\varsigma} \z)-\beta\K(\y^{\varsigma} + {\varsigma}\z) + \varsigma \z + \mathcal{P}\f, \\
			\y^{\varsigma}(0)& = \boldsymbol{\psi} - {\varsigma}\z(\omega,0)=:\y^{\varsigma}_0.
		\end{aligned}
		\right.
	\end{equation}
	Since $\z (\omega) \in C_{1/2} (\mathbb{R};\mathfrak{X}), $ then $\z (\omega,0)$ is a well defined element of $\H$. For each fixed $\omega\in\Omega$,   system \eqref{CTGF} is a deterministic system. Let us now provide the definition of a weak solution (in the deterministic sense, for each  fixed $\omega$) for \eqref{CTGF}. 
	\begin{definition}\label{defn-CTGF}
		Assume that $\y^{\varsigma}_0 \in \H$, $\z\in\mathrm{L}^2_{\mathrm{loc}}([0,\infty);\H)\cap\mathrm{L}^4_{\mathrm{loc}}([0,\infty);\mathbb{W}^{1,4}(\mathfrak{D}))$ and $\f\in \H^{-1}(\mathfrak{D})$. A function $\y^{\varsigma}$ is called a \textit{weak solution} of   system \eqref{CTGF} on the time interval $[0, \infty)$, if 
		\begin{align*}
			\y^{\varsigma} &\in  \mathrm{C}([0,\infty); \H) \cap \mathrm{L}^{2}_{\mathrm{loc}}(0,\infty; \V)\cap \mathrm{L}^{4}_{\mathrm{loc}}(0,\infty; \mathbb{W}^{1,4}(\mathfrak{D})), 
			\\
			\frac{\d\y^{\varsigma}}{\d t}&\in\mathrm{L}^{2}_{\mathrm{loc}}(0,\infty;\V') + \mathrm{L}^{\frac43}_{\mathrm{loc}}(0,\infty;\mathbb{W}^{-1,\frac43}(\mathfrak{D})),
		\end{align*}
		and it satisfies 
		\begin{itemize}
			\item [(i)] for any  $\boldsymbol{\phi}\in \X,$
			\begin{align*}
				&	\left<\frac{\d\y^{\varsigma}(t)}{\d t}, \boldsymbol{\phi}\right>
				 =  - \left\langle \nu \A\y^{\varsigma}(t)+\B(\y^{\varsigma}(t)+\varsigma\z(t))+\alpha\J(\y^{\varsigma}(t)+ {\varsigma}\z(t))+\beta\K(\y^{\varsigma}(t)+{\varsigma}\z(t)) - \varsigma\z(t)- \f , \boldsymbol{\phi} \right\rangle,
			\end{align*}
			for a.e. $t\in[0,\infty);$
			\item [(ii)] the initial data:
			$$\y^{\varsigma}(0)=\y^{\varsigma}_0 \ \text{ in }\ \H.$$
		\end{itemize}
	\end{definition}

	\begin{theorem}\label{solution-stochastic}
			Assume that \eqref{third-grade-paremeters-res} holds, $\v^{\varsigma}_0 \in \H$, $\f\in \H^{-1}(\mathfrak{D}) $ and $\z\in\mathrm{L}^2_{\mathrm{loc}}(\R;\H)\cap\mathrm{L}^4_{\mathrm{loc}}(\R;\mathbb{W}^{1,4}(\mathfrak{D}))$. Then, there exists a unique weak solution $\y^{\varsigma}(\cdot)$ to   system \eqref{CTGF} in the sense of Definition \ref{defn-CTGF}.
		\end{theorem}

	\begin{definition}
		We define a map $\Psi_{\varsigma} : [0,\infty) \times \Omega \times \H \to \H$ by
		\begin{align}
			(t, \omega, \boldsymbol{\psi}) \mapsto \y^{\varsigma}(t)  + \varsigma \z(\omega,t) \in \H,
		\end{align}
		where $\y^{\varsigma}(\cdot) $ is the unique weak solution to   system \eqref{CTGF} with the initial condition $\boldsymbol{\psi} - {\varsigma}\z(\omega,0).$
	\end{definition}

	 In fact, we have the following result for $\Psi_{\varsigma}$.

	\begin{theorem}[{\cite[Theorem 3.11]{Kinra+Cipriano_STGF}}]
		$(\Psi_{\varsigma}, \vartheta)$ is a random dynamical system.
	\end{theorem}
	Having established the necessary framework, we now present the weak solution  (in analytic sense) to system \eqref{STGF} with initial data $\x\in\H$ prescribed at time $s\in \R.$
	
	\begin{theorem}\label{STGF-Sol}
		In the framework of Definition \ref{Def_u}, suppose that $\v^{\varsigma}(t)=\y^{\varsigma}(t)+{\varsigma}\z(\omega,t), t\geq s,$ where $\y^{\varsigma}$ is the unique solution to system \eqref{CTGF} with initial data $\boldsymbol{\psi} - {\varsigma}\z(\omega,s)$ at time $s$. If the process $\{\v^{\varsigma}(t), \ t\geq s\},$ has trajectories in $\mathrm{C}([s, \infty); \H) \cap  \mathrm{L}^{4}_{\mathrm{loc}}([s, \infty); \mathbb{W}^{1,4}(\mathfrak{D}))$, then it is a solution to system \eqref{STGF}. Vice-versa, if a process $\{\v^{\varsigma}(t), t\geq s\},$ with trajectories in $\mathrm{C}([s, \infty); \H) \cap \mathrm{L}^{4}_{\mathrm{loc}}([s, \infty); \mathbb{W}^{1,4}(\mathfrak{D}))$ is a solution to system \eqref{STGF}, then a process $\{\y^{\varsigma}(t), t\geq s\},$ defined by $\y^{\varsigma}(t) = \v^{\varsigma}(t)- {\varsigma} \z(\omega,t), t\geq s,$ is a solution to \eqref{CTGF} on $[s, \infty).$
	\end{theorem}

We define, for $\varsigma\in(0,1]$, $\x \in \H,\ \omega \in \Omega,$ and $t\geq s,$
\begin{align}\label{combine_sol}
	\v_{\varsigma}(t, \omega ; s, \x) := \Psi_{\varsigma}(t-s, \vartheta_s \omega,\x) = \y^{\varsigma}\big(t, \omega; s, \x - \varsigma \z(\omega,s)\big) + \varsigma \z(\omega,t),
\end{align}
then the process $\{\v(t): \ t\geq s\}$ is a weak solution to  system $\eqref{STGF}$ in sense of Definition \ref{defn-CTGF}, for $t>s$, for each $s\in \mathbb{R}$ and each $\x \in \H$.

The following result is given in view of the study \cite{Kinra+Cipriano_STGF}, where the existence of random attractors for system \eqref{STGF} was established.
\begin{theorem}\label{thm-random-attractor}
	Let $\omega\in\Omega$. Assume that condition \eqref{third-grade-paremeters-res} is fulfilled and that Hypotheses \ref{assumpO}-\ref{assump1} hold. Let $\mathfrak{D}$ denote either a bounded or an unbounded domain, and suppose that $\f \in \H^{-1}(\mathfrak{D})$ in the former case or $\f \in \L^{2}(\mathfrak{D})$ in the latter. Then, the continuous random dynamical system  $\Psi_{\varsigma}(\cdot,\cdot,\cdot)$ associated with system \eqref{STGF} admits random attractors $\mathscr{A}_{\varsigma}(\omega)\in \mathfrak{DK}$.
\end{theorem}

%
%
%
%

		\subsection{Perturbation radius of the singleton attractor}
		Let us now consider the difference $\w_{\varsigma}(\cdot)=\y^{\varsigma}(\cdot)-\v(\cdot)$, where $\y^{\varsigma}(\cdot)$ and $\v(\cdot)$ are the unique weak solutions of the systems \eqref{CTGF} and \eqref{TGF}, respectively. We observe that $\w_{\varsigma}(\cdot)$ satisfies 
		\begin{equation}\label{Diff-TFG_add}
			\left\{
			\begin{aligned}
				\frac{\d\w_{\varsigma}}{\d t}&=-\nu \A\w_{\varsigma}-\B(\y^{\varsigma}+\varsigma\z) -\alpha\J(\y^{\varsigma} + {\varsigma} \z)-\beta\K(\y^{\varsigma} + {\varsigma}\z) + \varsigma \z   +\B(\v) +\alpha\J(\v) +\beta\K(\v)  , \\ 
				\w_{\varsigma}(0)&=-\varsigma\z(\omega,0),
			\end{aligned}
			\right.
		\end{equation}
		in $\X'$.

		The following lemma yields useful estimates for the weak solution of system \eqref{CTGF}, to be used later on.

		\begin{lemma}[{\cite[Lemma 4.1]{Kinra+Cipriano_STGF}}]\label{RA1}
		Assume that \eqref{third-grade-paremeters-res} holds and Hypotheses \ref{assumpO} and \ref{assump1} are satisfied. Suppose that $\y^{\varsigma}$ solves   system \eqref{CTGF} on the time interval $[a, \infty)$ with $\z \in  \mathrm{L}^2_{\mathrm{loc}}(\R; \H) \cap  \mathrm{L}^{4}_{\mathrm{loc}}(\R; \mathbb{W}^{1,4}(\mathfrak{D}))$ and $\f\in \H^{-1}(\mathfrak{D})$. Then, we have
		\begin{align}\label{Energy_esti1}
		&	\|\y^{\varsigma}(t)\|^2_{2} 
		\nonumber\\	& \leq 
			\|\y^{\varsigma}(s)\|^2_{2}\  e^{-\nu\lambda\left(1+ \frac{\varepsilon_0}{2}\right)(t-s)}  
			  + C\int_{s}^{t} \bigg[\|\z(\omega,\tau)\|^{2}_{2}+ \|\z(\omega,\tau)\|^{4}_{\mathbb{W}^{1,4}} + \|\f\|^2_{\H^{-1}} \bigg] e^{-\nu\lambda\left(1+ \frac{\varepsilon_0}{2}\right)(t-\tau)} \d \tau.
		\end{align}
	for any $t\geq \tau \geq a$.
	\end{lemma}
		
		We next state and prove a lemma that provides useful estimates for the weak solution of system \eqref{Diff-TFG_add}, which will play an important role in establishing the main result of this section, that is, the rate of convergence of random attractors $\mathscr{A}_{\varsigma}(\omega)$ towards deterministic singleton attractor.

		\begin{lemma}\label{PertRad-add}
			Assume that $\f\in \H^{-1}(\mathfrak{D})$, the condition \eqref{third-grade-paremeters-res} holds, and Hypotheses \eqref{assumpO}, \eqref{hypo-small-force} and \eqref{assump1} are fulfilled. Suppose that $\y^{\varsigma}(\cdot)$ and $\v(\cdot)$ are the unique solutions of \eqref{CTGF} and \eqref{TGF} corresponding to the initial data $\y^{\varsigma}(\omega,0)\in \mathfrak{DK}$ and $\x\in\H$, respectively. Then, for each $\omega\in\Omega$, there exists a random variable $\upgamma(\omega)$ independent of $\varsigma$ such that the $\w_{\varsigma}(\cdot)$ satisfies
			\begin{align}\label{Diff_rad}
				\limsup_{t\to \infty} \|\w_{\varsigma}(t,\vartheta_{-t}\omega;0,\w_{\varsigma}(\vartheta_{-t}\omega,0))\|^2_{2}\ \leq\ \varsigma^{\frac{4}{3}} \upgamma(\omega), \  \text{ for all }  \ \varsigma\in(0,1].
			\end{align}
		\end{lemma}
	\begin{proof}
	Firstly, let us choose $\omega\in\Omega$ and fix it. Secondly, note that $\varepsilon_0 = 1-\sqrt{\frac{\alpha^2}{2\beta\nu}}$ implies $\frac{\alpha^2}{4\nu(1-\varepsilon_0)} = \frac{\beta(1-\varepsilon_0)}{2}$. 	Next, taking the inner product with $\w_{\varsigma}$ to the equation $\eqref{Diff-TFG_add}_1$ and using \eqref{b0}, we have
		\begin{align}\label{Radius-1}
		&	\frac{1}{2}\frac{\d}{\d t} \|\w_{\varsigma}(t)\|^2_2 
		\nonumber\\ & = \langle-\nu \A\w_{\varsigma}(t)-\B(\y^{\varsigma}(t)+\varsigma\z(\omega,t)) -\alpha\J(\y^{\varsigma}(t) + {\varsigma} \z(\omega,t))-\beta\K(\y^{\varsigma}(t) + {\varsigma}\z(\omega,t)) + \varsigma \z(\omega,t) , \w_{\varsigma}(t)  \rangle  
		\nonumber\\ & \quad   +  \langle \B(\v(t)) +\alpha\J(\v(t)) +\beta\K(\v(t)) , \w_{\varsigma}(t)  \rangle  
		\nonumber\\ & = - \frac{\nu}{2}\|\Arm(\w_{\varsigma}(t))\|_2^2 - \underbrace{\varsigma b(\w_{\varsigma}(t),\w_{\varsigma}(t) , \z(\omega,t))}_{:= \wi\P_1(t)} - \underbrace{\varsigma^2 b(\z(\omega,t), \w_{\varsigma}(t)  ,\z(\omega,t))}_{:= \wi\P_2(t)} -  \underbrace{\varsigma b(\v(t),\w_{\varsigma}(t) ,\z(\omega,t)) }_{:= \wi\P_3(t)}
		 \nonumber\\ & \quad -   \underbrace{b(\w_{\varsigma}(t), \w_{\varsigma}(t), \v(t))}_{:= \wi\P_4(t)} 
		  - \underbrace{\varsigma b(\z(\omega,t),\w_{\varsigma}(t) ,\v(t)) }_{:= \wi\P_5(t)}
		 - \underbrace{\alpha \langle\J(\y^{\varsigma}(t)+\varsigma\z(\omega,t)) - \J(\v(t)+\varsigma\z(\omega,t)), \w_{\varsigma}(t)  \rangle }_{:= \wi\P_6(t)} \nonumber\\ & \quad - \underbrace{\alpha \langle\J(\v(t)+\varsigma\z(\omega,t)) - \J(\v(t)), \w_{\varsigma}(t)  \rangle }_{:= \wi\P_7(t)}
	-  \underbrace{\beta \langle\K(\y^{\varsigma}(t)+\varsigma\z(\omega,t)) - \K(\v(t)+\varsigma\z(\omega,t)), \w_{\varsigma}(t)  \rangle}_{:= \wi\P_8(t)}  
		\nonumber\\ & \quad	 -  \underbrace{\beta \langle\K(\v(t)+\varsigma\z(\omega,t)) - \K(\v(t)), \w_{\varsigma}(t)  \rangle}_{:= \wi\P_9(t)} + \underbrace{\varsigma(\z(\omega,t),\w_{\varsigma}(t))}_{:=\wi\P_{10}(t)} .
		\end{align}
	Using Cauchy-Schwarz, H\"older's and Young  inequalities, \eqref{Korn-ineq} and \eqref{Sobolev+Korn}, we get
	\begin{align}
		|\wi\P_1|& \leq  \varsigma \|\w_{\varsigma}\|_{2}\|\nabla\w_{\varsigma}\|_4\|\z\|_{4} \leq  \varsigma C \|\w_{\varsigma}\|_{2}\|\Arm(\w_{\varsigma})\|_4\|\z\|_{4}   
		\nonumber\\ & \leq    \frac{\varrho}{12} \|\w_{\varsigma}\|^2_{2} + \frac{\beta\varepsilon_0}{40}\|\Arm(\w_{\varsigma})\|^4_4  + \varsigma^4 C \|\z\|^4_{4},  \label{Radius-2}  \\
		|\wi\P_2|& \leq  \varsigma^2 \|\z\|_{2}\|\nabla\w_{\varsigma}\|_4\|\z\|_{4} \leq  \varsigma^2 C \|\z\|_{2}\|\Arm(\w_{\varsigma})\|_4\|\z\|_{4}   
		\nonumber\\ & \leq  \frac{\beta\varepsilon_0}{40}  \|\Arm(\w_{\varsigma})\|^4_4 + \varsigma^{\frac83} C[\|\z\|^2_{2} + \|\z\|^4_{4}], \label{Radius-3}  \\
		|\wi\P_3|& \leq  \varsigma \|\v\|_{2}\|\nabla\w_{\varsigma}\|_4\|\z\|_{4} \leq  \varsigma C \|\v\|_{2}\|\Arm(\w_{\varsigma})\|_4\|\z\|_{4}   
		\nonumber\\ & \leq   \frac{\beta\varepsilon_0}{40}\|\Arm(\w_{\varsigma})\|^4_4  +   \varsigma^{\frac43} C [\|\v\|^2_{2}  + \|\z\|^4_{4}], \label{Radius-4} \\
		|\wi\P_4|& \leq  \varsigma \|\w_{\varsigma}\|_{2}\|\nabla\w_{\varsigma}\|_2\|\v\|_{\infty} \leq   \frac{\mathrm{M}_d}{2} \|\w_{\varsigma}\|_{2}\|\Arm(\w_{\varsigma})\|_2\|\Arm(\v)\|_{4}   
		\nonumber\\ & \leq    \frac{(\mathrm{M}_d)^2}{4\nu\varepsilon_0}\|\Arm(\v)\|_{4}^2  \|\w_{\varsigma}\|^2_2  +  \frac{\nu\varepsilon_0}{4}\|\Arm(\w_{\varsigma})\|^2_{2},  \label{Radius-5} \\
		|\wi\P_5|& \leq  \varsigma \|\v\|_{2}\|\nabla\w_{\varsigma}\|_4\|\z\|_{4} \leq  \varsigma C \|\v\|_{2}\|\Arm(\w_{\varsigma})\|_4\|\z\|_{4}   
		\nonumber\\ & \leq    \frac{\beta\varepsilon_0}{40}\|\Arm(\w_{\varsigma})\|^4_4 +  \varsigma^{\frac43} C[\|\v\|^2_{2}   +  \|\z\|^4_{4}], \label{Radius-6}  \\
		|\wi\P_7| & \leq \frac{|\alpha|\varsigma}{2} | \langle \Arm(\z)\Arm(\v) + \Arm(\v)\Arm(\z) , \Arm(\w_{\varsigma}) \rangle| + \frac{|\alpha|\varsigma^2}{2} |\langle \Arm^2(\z) , \Arm(\w_{\varsigma}) \rangle |
		\nonumber\\ & 
		\leq \varsigma C \| \Arm(\z)\|_4 \|\Arm(\v)\|_4 \|\Arm(\w_{\varsigma})\|_2  + \varsigma^2 C \| \Arm(\z)\|_4^2 \| \Arm(\w_{\varsigma})\|_2
		\nonumber\\ & 
		\leq  \frac{\varrho}{24 \lambda} \|\Arm(\w_{\varsigma})\|^2_2  + \varsigma^2 C\left[ \| \Arm(\z)\|_4^4 +  \| \Arm(\v)\|_4^4\right],  \label{Radius-7}  \\
		|\wi\P_9| & \leq \frac{\beta\varsigma^3}{2} | \langle |\Arm(\z)|^2\Arm(\z) , \Arm(\w_{\varsigma}) \rangle|  + \frac{\beta\varsigma^2}{2} |\langle |\Arm(\z)|^2  \Arm(\v) + 2[\Arm(\z):\Arm(\v)]\Arm(\z) , \Arm(\w_{\varsigma}) \rangle|  
		\nonumber\\ & \quad +  \frac{\beta\varsigma}{2} |\langle 2[ \Arm(\z):\Arm(\v)] \Arm(\v) + |\Arm(\v)|^2\Arm(\z) , \Arm(\w_{\varsigma}) \rangle |
		\nonumber\\ & 
		\leq \frac{\beta\varsigma^3}{2}  \|\Arm(\z)\|^3_{4} \|\Arm(\w_{\varsigma})\|_4   + \frac{3\beta\varsigma^2}{2} \|\Arm(\z)\|^2_4  \|\Arm(\v)\|_4  \|\Arm(\w_{\varsigma})\|_4 + \frac{ 3 \beta\varsigma}{2} \|\Arm(\v)\|^2_4\|\Arm(\z)\|_4 \|\Arm(\w_{\varsigma})\|_4
		\nonumber\\ & 
		\leq \frac{\beta\varepsilon_0}{40} \|\Arm(\w_{\varsigma})\|^4_4  +  \varsigma^{\frac43} C\left[ \| \Arm(\z)\|_4^4 +  \| \Arm(\v)\|_4^4\right],  \label{Radius-8}  \\
		|\wi\P_{10}|& \leq \varsigma \|\z\|_2 \|\w_{\varsigma}\|_2\leq \varsigma^2C \|\z\|^2_{2} + \frac{\varrho}{12} \|\w_{\varsigma}\|^2_2   \label{Radius-9}.
	\end{align}
	From \cite[Equation (2.21)]{Hamza+Paicu_2007}, we have 
	\begin{align}\label{Radius-10}
		 |\wi\P_6| 
		& \leq \frac{\nu(1-\varepsilon_0)}{2}\|\Arm(\w_{\varsigma})\|_{2}^2 + \frac{\alpha^2}{4\nu(1-\varepsilon_0)} \int_{\mathfrak{D}} |\Arm(\w_{\varsigma}(x))|^2 ( |\Arm(\y^{\varsigma}(x)+\varsigma\z(x))|^2+|\Arm(\v(x)+\varsigma\z(x))|^2)\d x
		\nonumber\\ & = \frac{\nu(1-\varepsilon_0)}{2}\|\Arm(\w_{\varsigma})\|_{2}^2 + \frac{\beta(1-\varepsilon_0)}{2}\int_{\mathfrak{D}} |\Arm(\w_{\varsigma}(x))|^2 ( |\Arm(\y^{\varsigma}(x)+\varsigma\z(x))|^2+|\Arm(\v(x)+\varsigma\z(x))|^2)\d x.
	\end{align}
	Again, from \cite[Equation (2.13)]{Hamza+Paicu_2007}, we have 
	\begin{align}\label{Radius-11}
		  \wi\P_8 
		& = \frac{\beta}{2}\int_{\mathfrak{D}}( |\Arm(\y^{\varsigma}(x)+\varsigma\z(x))|^2-|\Arm(\v(x)+\varsigma\z(x))|^2)^2 \d x 
		\nonumber\\ & \quad + \frac{\beta}{2}\int_{\mathfrak{D}} |\Arm(\w_{\varsigma}(x))|^2 ( |\Arm(\y^{\varsigma}(x)+\varsigma\z(x))|^2+|\Arm(\v(x)+\varsigma\z(x))|^2)\d x.
	\end{align}
Also, we have 
\begin{align}\label{Radius-12}
	\frac{1}{2}\|\Arm(\w_{\varsigma})\|_{4}^4 \leq \int_{\mathfrak{D}} |\Arm(\w_{\varsigma}(x))|^2 ( |\Arm(\y^{\varsigma}(x)+\varsigma\z(x))|^2+|\Arm(\v(x)+\varsigma\z(x))|^2)\d x.
\end{align}
Here, we combine \eqref{Radius-1}-\eqref{Radius-12} and obtain
\begin{align}\label{Radius-13}
	&	\frac{1}{2}\frac{\d}{\d t} \|\w_{\varsigma}(t)\|^2_2 + \left(\frac{\nu\varepsilon_0}{4}- \frac{\varrho}{24\lambda}\right)\|\Arm(\w_{\varsigma}(t))\|_2^2 - \left[\frac{(\mathrm{M}_d)^2}{4\nu\varepsilon_0}\|\Arm(\v(t))\|_{4}^2 + \frac{\varrho}{6}\right]  \|\w_{\varsigma}(t)\|^2_2  + \frac{\beta\varepsilon_0}{8}\|\Arm(\w_{\varsigma}(t))\|_4^4 
	\nonumber\\ & \leq \varsigma^{\frac43} C \left[ \|\v(t)\|^2_{2} + \|\Arm(\v(t))\|^4_{4} +\|\z(\omega,t)\|^2_{2} + \|\z(\omega,t)\|^4_{\mathbb{W}^{1,4}}\right].
\end{align}
	Making use of the fact that $\frac{\nu\varepsilon_0}{4} > \frac{\varrho}{24\lambda}$ (see \eqref{C_2} above) and  Poincar\'e inequality \eqref{Poin} in \eqref{Radius-13}, we reach at
	\begin{align}\label{Radius-14}
		&	 \frac{\d}{\d t} \|\w_{\varsigma}(t)\|^2_2 +   \left[\nu\varepsilon_0 \lambda -  \frac{(\mathrm{M}_d)^2}{2\nu\varepsilon_0}\|\Arm(\v(t))\|_{4}^2 - \frac{\varrho}{2}\right]  \|\w_{\varsigma}(t)\|^2_2  + \frac{\beta\varepsilon_0}{8}\|\Arm(\w_{\varsigma}(t))\|_4^4 
		\nonumber\\ & \leq \varsigma^{\frac43} C \left[ \|\v(t)\|^2_{2} + \|\Arm(\v(t))\|^4_{4} +\|\z(\omega,t)\|^2_{2} + \|\z(\omega,t)\|^4_{\mathbb{W}^{1,4}}\right].
	\end{align}
	Since $0<\frac{\varrho}{2} < \nu^*\lambda$ (see \eqref{C_2} and \eqref{nu_ast} above), the time $\mathcal{T}_{\frac{\varrho}{2}}$ is well-defined (see Lemma \ref{lem-Absorb} above) and  we define a time $\mathscr{T}= \max\{\mathcal{T}_{\nu^*\lambda}, \mathcal{T}_{\frac{\varrho}{2}} \}$. Applying the variation of constant formula to \eqref{Radius-14} over time interval $[\mathscr{T}, t]$ with $t>\mathscr{T}+1$, and making use of \eqref{absorb2} and small forcing intensity condition \eqref{C_2}, we achieve
	\begin{align}\label{Radius-15}
	&	\|\w_{\varsigma}(t,\omega;0, \w_{\varsigma}(0))\|_2^2 
	\nonumber\\ & \leq
		e^{-\int_{\mathscr{T}}^{t}\big[\nu\varepsilon_0 \lambda -  \frac{(\mathrm{M}_d)^2}{2\nu\varepsilon_0}\|\Arm(\v(\tau))\|_{4}^2 - \frac{\varrho}{2}\big] \d \tau}\|\w_{\varsigma}(\mathscr{T},\omega;0, \w_{\varsigma}(0))\|_2^2
		\nonumber\\ & \quad + \varsigma^{\frac43}C \int_{\mathscr{T}}^{t}e^{-\int_{s}^{t}\big[\nu\varepsilon_0 \lambda -  \frac{(\mathrm{M}_d)^2}{2\nu\varepsilon_0}\|\Arm(\v(\tau))\|_{4}^2 - \frac{\varrho}{2}\big] \d \tau} \left[ \|\v(s)\|^2_{2} + \|\Arm(\v(s))\|^4_{4} +\|\z(\omega,s)\|^2_{2} + \|\z(\omega,s)\|^4_{\mathbb{W}^{1,4}}\right]\d s
		\nonumber\\ & \leq
		e^{-\int_{\mathscr{T}}^{t}\big[\nu\varepsilon_0 \lambda -  \frac{(\mathrm{M}_d)^2}{2\nu\varepsilon_0}\|\Arm(\v(\tau))\|_{4}^2 - \frac{\varrho}{2}\big] \d \tau}\|\w_{\varsigma}(\mathscr{T},\omega; 0 , \w_{\varsigma}(0))\|_2^2
		 + \varsigma^{\frac43}C e^{\int_{t}^{t+1}\big[\nu\varepsilon_0 \lambda -  \frac{(\mathrm{M}_d)^2}{2\nu\varepsilon_0}\|\Arm(\v(\tau))\|_{4}^2 - \frac{\varrho}{2}\big] \d \tau} 
		 \nonumber\\ & \quad \times \int_{\mathscr{T}}^{t}e^{-\int_{s}^{t+1}\big[\nu\varepsilon_0 \lambda -  \frac{(\mathrm{M}_d)^2}{2\nu\varepsilon_0}\|\Arm(\v(\tau))\|_{4}^2 - \frac{\varrho}{2}\big] \d \tau} \left[ \|\v(s)\|^2_{2} + \|\Arm(\v(s))\|^4_{4} +\|\z(\omega,s)\|^2_{2} + \|\z(\omega,s)\|^4_{\mathbb{W}^{1,4}}\right]\d s
		\nonumber\\ & \leq
		e^{-\frac{\varrho}{2}(t-\mathscr{T})}\|\w_{\varsigma}(\mathscr{T},\omega; 0 , \w_{\varsigma}(0))\|_2^2
		+ \varsigma^{\frac43}C e^{\nu\varepsilon_0 \lambda }  \int_{\mathscr{T}}^{t}e^{-(\nu\varepsilon_0\lambda- \frac{\varrho}{2})(t+1-s) + \frac{(\mathrm{M}_d)^2}{2\nu\varepsilon_0} \int_{s}^{t+1}  \|\Arm(\v(\tau))\|_{4}^2 \d \tau}
		\nonumber\\ & \quad  \times \left[ \|\v(s)\|^2_{2} + \|\Arm(\v(s))\|^4_{4} +\|\z(\omega,s)\|^2_{2} + \|\z(\omega,s)\|^4_{\mathbb{W}^{1,4}}\right]\d s
		\nonumber\\ & \leq
		e^{-\frac{\varrho}{2}(t-\mathscr{T})}\|\w_{\varsigma}(\mathscr{T},\omega;0, \w_{\varsigma}(0))\|_2^2
		 + \varsigma^{\frac43}C  \int_{\mathscr{T}}^{t}e^{-(\nu\varepsilon_0\lambda- \frac{\varrho}{2})(t+1-s) + \frac{(\mathrm{M}_d)^2}{2\nu\varepsilon_0} \left( \frac{1}{\beta^*\nu^{*}}\right)^{\frac12} \left( \frac{1}{2} + \frac{1}{\nu^*\lambda} \right)^{\frac12}\|\f\|_{\H^{-1}}(t+1-s) }  
		 \nonumber\\ & \quad \times \left[ \|\v(s)\|^2_{2} + \|\Arm(\v(s))\|^4_{4} +\|\z(\omega,s)\|^2_{2} + \|\z(\omega,s)\|^4_{\mathbb{W}^{1,4}}\right]\d s
		 \nonumber\\ & \leq
		 e^{-\frac{\varrho}{2}(t-\mathscr{T})}\|\w_{\varsigma}(\mathscr{T}, \omega;0, \w_{\varsigma}(0))\|_2^2
		 \nonumber\\ & \quad + \varsigma^{\frac43}C  \int_{\mathscr{T}}^{t}e^{- \frac{\varrho}{2}(t+1-s) }   \left[ \|\v(s)\|^2_{2} + \|\Arm(\v(s))\|^4_{4} +\|\z(\omega,s)\|^2_{2} + \|\z(\omega,s)\|^4_{\mathbb{W}^{1,4}}\right]\d s
		 \nonumber\\ & \leq 
		 e^{-\frac{\varrho}{2}(t-\mathscr{T})}\|\w_{\varsigma}(\mathscr{T}, \omega, \w_{\varsigma}(0))\|_2^2 + \varsigma^{\frac43}C  \int_{\mathscr{T}}^{t}e^{- \frac{\varrho}{2}(t-s) }   \left[ \|\v(s)\|^2_{2} + \|\Arm(\v(s))\|^4_{4} \right]\d s
		 \nonumber\\ & \quad + \varsigma^{\frac43}C   \int_{\mathscr{T}}^{t}e^{- \frac{\varrho}{2}(t-s) }   \left[ \|\z(\omega,s)\|^2_{2} + \|\z(\omega,s)\|^4_{\mathbb{W}^{1,4}}\right]\d s .
	\end{align}
	
	Using \eqref{absorb1}, we note that for all $t>\mathscr{T}+1$
	\begin{align}\label{Radius-16}
		\int_{\mathscr{T}}^{t}e^{- \frac{\varrho}{2}(t-s) }   \left[ \|\v(s)\|^2_{2} + \|\Arm(\v(s))\|^4_{4} \right]\d s  & \leq  \frac{2\|\f\|_{\H^{-1}}^2}{\nu^*\varrho} \int_{\mathscr{T}}^{t}e^{- \frac{\varrho}{2}(t-s) }    \d s +  \frac{2\|\f\|_{\H^{-1}}^2}{\nu^*\varrho}
		\nonumber\\ & \leq  \frac{4\|\f\|_{\H^{-1}}^2}{\nu^*(\varrho)^2} +  \frac{2\|\f\|_{\H^{-1}}^2}{\nu^*\varrho} < + \infty.
	\end{align}
	Also, we have from \eqref{stationary} and Lemma \ref{Bddns5} that for all $t>\mathscr{T}+1$ and for each $\omega\in\Omega$
	\begin{align}\label{Radius-17}
		\int_{\mathscr{T}}^{t}e^{- \frac{\varrho}{2}(t-s) }   \left[ \|\z(\vartheta_{-t}\omega,s)\|^2_{2} + \|\z(\vartheta_{-t}\omega,s)\|^4_{\mathbb{W}^{1,4}}\right]\d s & = \int_{\mathscr{T}}^{t}e^{- \frac{\varrho}{2}(t-s) }   \left[ \|\z(\omega,s-t)\|^2_{2} + \|\z(\omega,s-t)\|^4_{\mathbb{W}^{1,4}}\right]\d s 
		\nonumber\\ &  =  \int_{\mathscr{T}-t}^{0}e^{ \frac{\varrho\tau}{2}}   \left[ \|\z(\omega,\tau)\|^2_{2} + \|\z(\omega,\tau)\|^4_{\mathbb{W}^{1,4}}\right]\d \tau
		\nonumber\\ & \leq \int_{-\infty}^{0}e^{ \frac{\varrho\tau}{2}}   \left[ \|\z(\omega,\tau)\|^2_{2} + \|\z(\omega,\tau)\|^4_{\mathbb{W}^{1,4}}\right]\d \tau < + \infty.
	\end{align}
	
	From Lemmas \ref{lem-Absorb} and \ref{RA1}, \eqref{stationary} and the fact that $\frac{\varrho}{4} < \nu\lambda\left(1+ \frac{\varepsilon_0}{2}\right)$, for $t>\mathscr{T}+1$, we have
	\begin{align}\label{Radius-18}
		& e^{-\frac{\varrho}{2}(t-\mathscr{T})} \|\w_{\varsigma}(\mathscr{T}, \vartheta_{-t}\omega; 0, \w_{\varsigma}(0))\|_2^2 
		\nonumber\\ &\leq 2 e^{-\frac{\varrho}{2}(t-\mathscr{T})} \left[\|\y^{\varsigma}(\mathscr{T}, \vartheta_{-t}\omega;0, \y^{\varsigma}(0))\|_2^2 + \|\v(\mathscr{T})\|_2^2\right]
		\nonumber\\ &\leq 2  e^{-\frac{\varrho}{2}(t-\mathscr{T})} \bigg\{\|\y^{\varsigma}(\vartheta_{-t}\omega, 0)\|^2_{2}\  e^{-\nu\lambda\left(1+ \frac{\varepsilon_0}{2}\right)\mathscr{T}}  
		\nonumber\\ & \quad + C\int_{0}^{\mathscr{T}} \bigg[\|\z(\vartheta_{-t}\omega,\tau)\|^{2}_{2}+ \|\z(\vartheta_{-t}\omega,\tau)\|^{4}_{\mathbb{W}^{1,4}} + \|\f\|^2_{\H^{-1}} \bigg] e^{-\nu\lambda\left(1+ \frac{\varepsilon_0}{2}\right)(\mathscr{T}-\tau)} \d \tau 
		 + \frac{\|\f\|_{\H^{-1}}^2}{(\nu^*)^2\lambda}\bigg\}
		 \nonumber\\ &\leq 2  e^{-\frac{\varrho}{4}(t-\mathscr{T})} \bigg\{\|\y^{\varsigma}(\vartheta_{-t}\omega, 0)\|^2_{2}  e^{-\frac{\varrho}{4}(t-\mathscr{T})} 
		 \nonumber\\ & \quad + C e^{-\frac{\varrho}{4}(t-\mathscr{T})} \int_{0}^{\mathscr{T}} \bigg[\|\z(\omega,\tau-t)\|^{2}_{2}+ \|\z(\omega,\tau-t)\|^{4}_{\mathbb{W}^{1,4}} + \|\f\|^2_{\H^{-1}} \bigg] e^{-\frac{\varrho}{4}(\mathscr{T}-\tau)} \d \tau 
		 + \frac{\|\f\|_{\H^{-1}}^2}{(\nu^*)^2\lambda}\bigg\}
		 \nonumber\\ &\leq 2  e^{-\frac{\varrho}{4}(t-\mathscr{T})} \bigg\{\|\y^{\varsigma}(\vartheta_{-t}\omega, 0)\|^2_{2}   
		 \nonumber\\ & \quad + C  \int_{0}^{\mathscr{T}} \bigg[\|\z(\omega,\tau-t)\|^{2}_{2}+ \|\z(\omega,\tau-t)\|^{4}_{\mathbb{W}^{1,4}} + \|\f\|^2_{\H^{-1}} \bigg] e^{-\frac{\varrho}{4}(t-\tau)} \d \tau 
		 + \frac{\|\f\|_{\H^{-1}}^2}{(\nu^*)^2\lambda}\bigg\}
		 \nonumber\\ &\leq 2  e^{-\frac{\varrho}{4}(t-\mathscr{T})} \bigg\{\|\y^{\varsigma}(\vartheta_{-t}\omega, 0)\|^2_{2}    
		 + C  \int_{-t}^{\mathscr{T}-t} \bigg[\|\z(\omega,s)\|^{2}_{2}+ \|\z(\omega,s)\|^{4}_{\mathbb{W}^{1,4}} + \|\f\|^2_{\H^{-1}} \bigg] e^{\frac{\varrho}{4}s} \d s 
		 + \frac{\|\f\|_{\H^{-1}}^2}{(\nu^*)^2\lambda}\bigg\}
		 \nonumber\\ &\leq 2  e^{-\frac{\varrho}{4}(t-\mathscr{T})} \bigg\{\|\y^{\varsigma}(\vartheta_{-t}\omega, 0)\|^2_{2}   
		 + C  \int_{-\infty}^{0} \bigg[\|\z(\omega,s)\|^{2}_{2}+ \|\z(\omega,s)\|^{4}_{\mathbb{W}^{1,4}} + \|\f\|^2_{\H^{-1}} \bigg] e^{\frac{\varrho}{4}s} \d s 
		 + \frac{\|\f\|_{\H^{-1}}^2}{(\nu^*)^2\lambda}\bigg\}
		 \nonumber\\ & \to 0 \text{	as } t\to\infty,
	\end{align}
where we have used Lemmas \ref{Bddns4} and \ref{Bddns5} and the fact that $ \y^{\varsigma}(\omega,0)\in \mathfrak{DK}$ to pass to the limit as $t\to\infty$. Now, we replace $\omega$ with $\vartheta_{-t}\omega$ in \eqref{Radius-15}, make use of \eqref{Radius-16}-\eqref{Radius-17}, and pass to the upper limit as $t\to\infty$ in view of \eqref{Radius-18} to get
	\begin{align*}
	& \limsup_{t\to \infty}	\|\w_{\varsigma}(t,\vartheta_{-t}\omega; 0 , \w_{\varsigma}(0))\|_2^2 
	  \leq 
	 \varsigma^{\frac43}C \bigg[\frac{4\|\f\|_{\H^{-1}}^2}{\nu^*(\varrho)^2} +  \frac{2\|\f\|_{\H^{-1}}^2}{\nu^*\varrho}  +  \int_{-\infty}^{0}e^{ \frac{\varrho\tau}{2}}   \left[ \|\z(\omega,\tau)\|^2_{2} + \|\z(\omega,\tau)\|^4_{\mathbb{W}^{1,4}}\right]\d \tau\bigg],
\end{align*}
which concludes the proof.
\end{proof}

	\begin{theorem}\label{Conver-add}
		Assume that the condition \eqref{third-grade-paremeters-res} holds, and Hypotheses \eqref{assumpO}, \eqref{hypo-small-force} and \eqref{assump1} are satisfied.  Let $\mathfrak{D}$ denote either a bounded or an unbounded domain, and suppose that $\f \in \H^{-1}(\mathfrak{D})$ in the former case or $\f \in \L^{2}(\mathfrak{D})$ in the latter. Then, there exists a random variable $R(\omega)$ (independent of $\varsigma$) such that for every $\varsigma\in(0,1],$  random attractors $\mathscr{A}_{\varsigma}$ (obtained in Theorem \ref{thm-random-attractor}) and the deterministic attractor $\mathscr{A}=\{\emph{\textbf{a}}_{*}\}$ (obtained in Theorem \ref{D-SA}) satisfy
	\begin{align*}
		\emph{dist}_{\H}(\mathscr{A}_{\varsigma}(\omega),\mathscr{A}) \leq \varsigma^{\frac{2}{3}} R(\omega), \ \text{ for all }\ \omega\in \Omega,
	\end{align*}
	where $\emph{dist}_{\H}(A,B)=\max\{\emph{dist}(A,B),\emph{dist}(B,A)\}$.
\end{theorem}
\begin{proof}
	Let $\emph{\textbf{a}}_{\varsigma}(\omega)$ be an arbitrary element of $\mathscr{A}_{\varsigma}(\omega)$. Then by the invariance property of random attractor, we have $\emph{\textbf{a}}_{\varsigma}(\omega)=\v_{\varsigma}(t,\vartheta_{-t}\omega; 0 ,\emph{\textbf{a}}_{\varsigma}(\vartheta_{-t}\omega)),$ for some $\emph{\textbf{a}}_{\varsigma}(\vartheta_{-t}\omega)\in\mathscr{A}_{\varsigma}(\vartheta_{-t}\omega), t\geq 0.$ From Theorem \ref{D-SA}, we know that the deterministic attractor $\mathscr{A}=\{\emph{\textbf{a}}_{*}\}$ is a singleton. Therefore, we have 
	\begin{align*}
		\|\emph{\textbf{a}}_{\varsigma}(\omega)-\emph{\textbf{a}}_{*}\|_{2} & = \|\v_{\varsigma}(t,\vartheta_{-t}\omega;0,\emph{\textbf{a}}_{\varsigma}(\vartheta_{-t}\omega))-\v(t,\emph{\textbf{a}}_{*})\|_{2}
		\nonumber\\ & \leq  \|\y^{\varsigma}(t,\vartheta_{-t}\omega; 0 ,\emph{\textbf{a}}_{\varsigma}(\vartheta_{-t}\omega)- \z(\vartheta_{-t}\omega,0))-\v(t,\emph{\textbf{a}}_{*})\|_{2} + \varsigma \|\z(\omega,0)\|_2 
		\nonumber\\ & \leq  \|\y^{\varsigma}(t,\vartheta_{-t}\omega;0,\emph{\textbf{a}}_{\varsigma}(\vartheta_{-t}\omega)- \z(\vartheta_{-t}\omega,0)) -\v(t,\emph{\textbf{a}}_{*})\|_{2} + \varsigma \|\z(\omega,0)\|_2, \ \text{ for all }\ t\geq 0,
	\end{align*}
which gives by Lemma \ref{PertRad-add}
\begin{align*}
	\|\emph{\textbf{a}}_{\varsigma}(\omega)-\emph{\textbf{a}}_{*}\|_{2} \leq \varsigma^{\frac23}[ \upgamma(\omega) +  \|\z(\omega,0)\|_2].
\end{align*}
This concludes the proof.
\end{proof}
\begin{remark}\label{rem-ROC}
	The primary objective of this work is to establish the rate of convergence  of the random attractor towards the deterministic singleton attractor. According to Theorem \ref{Conver-add}, the obtained rate of convergence is of order $\varsigma^{\frac{2}{3}}$. This rate differs from that of the stochastic Navier-Stokes equations with additive noise, where the convergence order is typically $\varsigma$ (see \cite{HCPEK}).

\end{remark}

\vskip 2mm
\noindent
\textbf{Acknowledgments:}  “This work is funded by national funds through the FCT-Fundação para a Ciência e a Tecnologia, I.P., under the scope of the projects UIDB/00297/2020 (https://doi.org/ 10.54499/UIDB/00297/ 2020) and UIDP/00297/2020 (https://doi.org/10.54499/UIDP/00297/ 2020) (Center for Mathematics and Applications)”. K. Kinra would like to thank Prof. Fernanda Cipriano, Center for Mathematics and Applications (NOVA Math) and Department of Mathematics, NOVA School of Science and Technology (NOVA FCT),	Portugal, for useful discussions. K. Kinra would like to thank Prof. Manil T. Mohan, Department of Mathematics, Indian Institute of Technology Roorkee, Roorkee, India, for introducing him into the research area of attractor theory.

\medskip\noindent
\textbf{Data availability:} No data was used for the research described in the article.

\medskip\noindent
\textbf{Declarations}: During the preparation of this work, the authors have not used AI tools.

\medskip\noindent
\textbf{Conflict of interest:} The authors declare no conflict of interest.

\end{document}